\documentclass[11pt,a4paper]{amsart}

\usepackage[margin=2.3cm]{geometry}
\usepackage{amsmath,amssymb,amsthm,mathtools,xcolor}
\numberwithin{equation}{section}
\numberwithin{figure}{section}
\usepackage{bm}
\usepackage[T1]{fontenc}
\usepackage[expansion=false]{microtype}
\usepackage{enumitem}
\usepackage{booktabs}
\usepackage{graphicx}
\usepackage[normalem]{ulem}
\usepackage[hidelinks]{hyperref}
\usepackage{aliascnt}
\usepackage[nameinlink,capitalize]{cleveref}

\makeatletter
\def\filename{\texttt{\jobname.tex}} 
\makeatother

\hypersetup{
  pdftitle={Functional Limits and Separation Times for Two Interacting ERWs},
  pdfauthor={Rafik Aguech, Shuhei Shibata, Tomoyuki Shirai}
}

\newtheorem{theorem}{Theorem}[section]

\newaliascnt{proposition}{theorem}
\newtheorem{proposition}[proposition]{Proposition}
\aliascntresetthe{proposition}

\newaliascnt{lemma}{theorem}
\newtheorem{lemma}[lemma]{Lemma}
\aliascntresetthe{lemma}

\newaliascnt{corollary}{theorem}
\newtheorem{corollary}[corollary]{Corollary}
\aliascntresetthe{corollary}

\newaliascnt{remark}{theorem}
\newtheorem{remark}[remark]{Remark}
\aliascntresetthe{remark}

\newaliascnt{definition}{theorem}

\aliascntresetthe{definition}

\crefname{theorem}{theorem}{theorems}
\Crefname{theorem}{Theorem}{Theorems}

\crefname{proposition}{proposition}{propositions}
\Crefname{proposition}{Proposition}{Propositions}

\crefname{lemma}{lemma}{lemmas}
\Crefname{lemma}{Lemma}{Lemmas}

\crefname{corollary}{corollary}{corollaries}
\Crefname{corollary}{Corollary}{Corollaries}

\crefname{remark}{remark}{remarks}
\Crefname{remark}{Remark}{Remarks}

\crefname{definition}{definition}{definitions}
\Crefname{definition}{Definition}{Definitions}

\newcommand{\R}{\mathbb{R}}

\newcommand{\Z}{\mathbb{Z}}
\newcommand{\N}{\mathbb{N}}
\newcommand{\E}{\mathbb{E}}
\newcommand{\Prob}{\mathbb{P}}
\newcommand{\F}{\mathcal{F}}
\newcommand{\tr}{\operatorname{tr}}
\newcommand{\HS}{\mathrm{HS}}
\DeclareMathOperator{\diag}{diag}
\def\dd{\mathbf{e}}

\title[Functional Limits and Separation Times for Two Interacting ERWs]
{Functional Limits and Separation Times for Two Interacting Elephant Random Walks}

\author{Rafik Aguech}
\address{Department of Statistics and Operations Research, College of Science,
King Saud University, Riyadh, Saudi Arabia}
\email{raguech@ksu.edu.sa}

\author{Shuhei Shibata}
\address{Joint Graduate School of Mathematics for Innovation, Kyushu University, Fukuoka, Japan}
\email{shibata.shuhei.746@s.kyushu-u.ac.jp}

\author{Tomoyuki Shirai}
\address{Institute of Mathematics for Industry, Kyushu University, Fukuoka, Japan}
\email{shirai@imi.kyushu-u.ac.jp}


\begin{document}

\begin{abstract}
We establish functional scaling limits and study first separation times for the interacting two
elephant model studied by Aguech and Qin.
In the joint diffusive regime, we give a direct martingale proof of convergence
to a two-dimensional continuous Gaussian process represented by a matrix-kernel
analogue of the noise-reinforced Brownian motion.
We then investigate the difference process, whose diffusive scaling persists
in the symmetric case even when the joint walk is critical or superdiffusive.
For the first separation time of the two walks, we establish convergence
in distribution to the first exit time of the limiting Gaussian process, together with convergence of all positive moments under diffusive
scaling.
We also obtain monotonicity results for the limiting exit time by
combining explicit covariance identities with Anderson's inequality.
\end{abstract}

\maketitle

\begin{NoHyper}
\begingroup
\renewcommand{\thefootnote}{}
\endgroup
\end{NoHyper}

\tableofcontents

\section{Introduction}
\label{sec:intro}

The Elephant Random Walk (ERW), introduced by Sch\"utz and Trimper~\cite{schutz2004elephants},
is a discrete-time random walk on $\Z$ with long-range memory, whose increments depend on the whole past trajectory.
At each step, a previous time is sampled uniformly and the corresponding increment is repeated with probability $p$ or reversed with probability $1-p$.
Writing $\alpha=2p-1\in[-1,1]$ for the reinforcement parameter, the model exhibits three different asymptotic regimes:
the diffusive regime $p<3/4$ (equivalently $\alpha<1/2$),
the critical regime $p=3/4$ (equivalently $\alpha=1/2$),
and the superdiffusive regime $p>3/4$ (equivalently $\alpha>1/2$).
Its asymptotic behaviour has been extensively investigated from various perspectives, see \cite{bercu2018martingale, bercu2019multidimensional, coletti2017strong, coletti2017central, guerin2025limit, guerin2026fixed, kubota2019gaussian}.

In the reinforcing part of the diffusive regime, the rescaled position converges to the continuous Gaussian process
\[
Z_t=\int_0^t\left(t/s\right)^\alpha\,dB_s,
\qquad
0<\alpha<\frac12,
\]
known as the \emph{noise-reinforced Brownian motion}.
This process was studied systematically by Bertoin~\cite{bertoin2020universality} and arises as a universal Gaussian scaling limit for a broad class of reinforced stochastic processes. Related integral representations in the negatively reinforced setting appear
as \emph{noise-counterbalanced Brownian motions} in
Bertenghi and Rosales-Ortiz~\cite{bertenghi2022joint}.
Functional limit theorems in the Skorokhod space have been established for the one-dimensional ERW by Baur--Bertoin~\cite{baur2016elephant}, for the multi-dimensional ERW by Bertenghi~\cite{bertenghi2022functional}, and for elephant walks on periodic structures by Shibata~\cite{shibata2025functional}. 

Recently, the study of interacting reinforced random walks has also attracted considerable attention, see~\cite{chen2014two, erhard2024stochastic, aletti2017synchronization, prado2023two, rosales2022vertex, gantert2024interacting, prado2025interacting, das2024elephant, marquioni2019multidimensional, aguech2026two}.
In the context of elephant random walks,
Marquioni~\cite{marquioni2019multidimensional} introduced an interacting two-elephant model and investigated mainly the first- and second-order moments of the joint process.
More recently, Das~\cite{das2024elephant} studied a more general system of multiple interacting elephant random walks, allowing the memory interactions among the walks to be specified through a graph. 
Aguech and Qin~\cite{aguech2026two} studied an interacting two-elephant model in which each elephant samples a uniformly chosen increment from the history of the other elephant.
Writing $\alpha_i=2p_i-1$ for the reinforcement parameter of the $i$-th elephant, $i=1,2$, they established, among other results, a fixed-time bivariate central limit theorem in the diffusive regime $\alpha_1\alpha_2<1/4$.
Our first aim is to strengthen this fixed-time convergence to a functional limit theorem and to investigate the probabilistic structure of the limiting Gaussian process.

More precisely, let
\[
  S_n=\bigl(S_n^{(1)},S_n^{(2)}\bigr)^\top,\qquad A=\begin{pmatrix}0&\alpha_1\\\alpha_2&0\end{pmatrix}.
\]
We prove that, whenever $\alpha_1\alpha_2<1/4$,
\[
\left\{
\frac{S_{[nt]}}{\sqrt n}
\right\}_{t\ge0}
\Longrightarrow
\{Z_t\}_{t\ge0}
\qquad
\text{as $n\to\infty$},
\]
where
\[
Z_t
=
\int_0^t
(t/s)^A\,dB_s.
\]
Here and throughout, $\Longrightarrow$ denotes convergence in distribution, and $x^A$ denotes $e^{(\log x)A}$ for $x>0$.
Thus the joint scaling limit is a natural matrix-kernel analogue of the noise-reinforced Brownian
motion. We derive explicit representations and covariance identities for this process and describe its
stochastic integral equation and Lamperti transform. The proof is based on a direct martingale
argument adapted to the two-elephant structure.

Our second aim is to study the difference process
\[
  \mathcal D_n:=S_n^{(1)}-S_n^{(2)}.
\]
We establish a functional limit theorem and show
that, in the symmetric case $\alpha_1=\alpha_2=\alpha$, diffusive scaling
persists throughout $-1/2<\alpha<1$, even when the joint walk is critical or superdiffusive. We derive the covariance of the limiting Gaussian difference, and we prove a global and a local law of the iterated logarithm for that process.

This functional limit naturally leads to a separation-time problem. For the first separation time
\[
  \sigma_c:=\inf\{n\ge1:|\mathcal D_n|\ge2c\},
\]
we show that
\[
  \frac{\sigma_c}{c^2}
  \Longrightarrow
  \tau_{\alpha_1,\alpha_2}
  :=
  \inf\{t>0:|V_t|\ge2\}
  \qquad
\text{as $c\to\infty$},
\]
where $V$ is the limiting Gaussian process. Moreover, we obtain, for every $q>0$,
\[
  \E[\sigma_c^q]
  =
  c^{2q}\E[\tau_{\alpha_1,\alpha_2}^q]
  +o(c^{2q}).
\]
We further establish monotonicity properties of the limiting exit times
with respect to the usual stochastic order, both along product level sets
$\alpha_1\alpha_2$ and along the equal-parameter diagonal
$\alpha_1=\alpha_2$.

Stopping times associated with a single ERW have also been investigated recently. Andr\'e and Zuazn\'abar~\cite{andre2026estimates} studied the
escape time of a single elephant random walk from a symmetric interval,
obtaining two-sided exponential tail bounds in the diffusive regime and
quadratic asymptotics for its expectation.
Ram\'irez-Gonz\'alez and Machado~\cite{ramirez2026frog} investigated
related escape-time asymptotics across the diffusive, critical, and
superdiffusive regimes in connection with a frog model with elephant
random walk particle motions.
Related questions concerning cover times and ranges were studied by
Qin~\cite{qin2026cover}.

The present model is contained in the graph-based framework of
Das~\cite{das2024elephant}: functional scaling limits follow from the
strong Gaussian approximation for general interaction matrices in
\cite[Theorem~13]{das2024elephant}, while diffusive projections in the
superdiffusive regime and the symmetric two-elephant case are treated in
\cite[Section~3.1.2 and Corollary~15]{das2024elephant}, respectively.
Our direct martingale approach gives these functional limits in the present
setting and further allows us to study the first separation time, including
its distributional and moment asymptotics and stochastic monotonicity
properties of the limiting exit time.
We also note that, in a different setting, Roy, Takei and
Tanemura~\cite{roy2024often} considered a pair of \emph{independent}
elephant random walks with a common memory parameter and established a
sharp criterion for whether the two walkers collide infinitely or only
finitely many times. Subsequently, Shibata and
Shirai~\cite{shibata2025remark} extended these results to the case of
distinct memory parameters.

We do not pursue joint functional limit theorems in the critical and superdiffusive regimes in this paper.
The corresponding fixed-time asymptotics are discussed by Aguech and Qin~\cite{aguech2026two}.

The outline of the paper is as follows.
In \Cref{sec:setup}, we introduce the interacting ERW model and collect
some preliminary facts on the reinforcement matrix.
In \Cref{sec:main}, we state the joint functional limit theorem and study
the resulting Gaussian process.
In \Cref{sec:proof}, we prove the functional limit theorem by a martingale
approach.
In \Cref{sec:difference}, we establish the functional limit theorem for the
difference process and analyze its limiting Gaussian process.
Finally, in \Cref{sec:exit}, we investigate the first separation time
of the two walks, including moment asymptotics and monotonicity properties
of the limiting exit time.

\section{Setup}
\label{sec:setup}

Throughout, let $i\in\{1,2\}$.
Let $\{\xi_n^{(i)}\}_{n\ge2}$ be a sequence of i.i.d.\ Rademacher random variables with
parameter $p_i\in[0,1]$, that is
\[
  \Prob(\xi_n^{(i)}=1)=1-\Prob(\xi_n^{(i)}=-1)=p_i.
\]
Let $\{U_n^{(i)}\}_{n\ge1}$ be a sequence of independent random variables such that, for each
$n\ge1$, $U_n^{(i)}$ is uniformly distributed on $\{1,2,\dots,n\}$.
Assume that all random variables
$\{\xi_n^{(i)}:n\ge2,\,i=1,2\}$ and $\{U_n^{(i)}:n\ge1,\,i=1,2\}$ are mutually independent.

We define two interacting elephant random walks $\{S_n^{(i)}\}_{n\ge0}$ on $\Z$, following the model studied by Aguech and Qin~\cite{aguech2026two}.
Let $(X_1^{(1)},X_1^{(2)})$ be an arbitrary
$\{-1,1\}^2$-valued random vector, independent of $\{\xi_n^{(i)}:n\ge2,\ i=1,2\}$ and $\{U_n^{(i)}:n\ge1,\ i=1,2\}$. (The choice of the initial distribution does not affect the limiting laws or the leading‑-order moment asymptotics below.)
For each $n\ge1$, given the past steps $\{X_k^{(1)}\}_{k=1}^n$ and $\{X_k^{(2)}\}_{k=1}^n$,
the $(n+1)$-th steps are defined by
\begin{equation}
\label{eq:step-interaction}
  X_{n+1}^{(1)}=\xi_{n+1}^{(1)}\,X_{U_n^{(2)}}^{(2)},
  \qquad
  X_{n+1}^{(2)}=\xi_{n+1}^{(2)}\,X_{U_n^{(1)}}^{(1)}.
\end{equation}
Set
\[
  \F_n=\sigma\bigl(X_1^{(1)},\dots,X_n^{(1)},\,X_1^{(2)},\dots,X_n^{(2)}\bigr),
  \qquad 
  \mathcal F_0:=\{\emptyset,\Omega\}.
\]
Then, for every $n\in\N$, the random variables
$\xi_{n+1}^{(1)}$, $\xi_{n+1}^{(2)}$, $U_n^{(1)}$, and $U_n^{(2)}$ are independent of $\F_n$.
The interacting elephant random walks are then defined by
\[
  S_0^{(i)}=0
  \qquad\text{and}\qquad
  S_n^{(i)}=\sum_{k=1}^n X_k^{(i)},\qquad n\ge1.
\]
We write
\begin{equation}
\label{eq:ERWs-vector}
  S_n:=\begin{pmatrix}S_n^{(1)}\\ S_n^{(2)}\end{pmatrix}.
\end{equation}

The parameters $p_i$ are the memory parameters of the first and second elephants.
We introduce the reinforcement matrix and parameters
\[
A:=\begin{pmatrix}0&\alpha_1\\\alpha_2&0\end{pmatrix},\qquad\alpha_i=2p_i-1\in[-1,1].
\]
Set 
\[
\mu:=\alpha_1\alpha_2\in [-1,1].
\]
The power-series expansion of the matrix exponential gives
\begin{equation}
    \label{eq:e-matrix-function}
    e^{xA}=c_\mu(x)I+s_\mu(x)A,
\end{equation}
where
\begin{equation}
    \label{function:convenient_e-matrix}
    c_\mu(x):=
\begin{cases}
\cosh(\sqrt{\mu}\,x), & \mu>0,\\
1, & \mu=0,\\
\cos(\sqrt{-\mu}\,x), & \mu<0,
\end{cases}
\qquad 
s_\mu(x):=
\begin{cases}
\dfrac{\sinh(\sqrt{\mu}\,x)}{\sqrt{\mu}}, & \mu>0,\\[2mm]
x, & \mu=0,\\[2mm]
\dfrac{\sin(\sqrt{-\mu}\,x)}{\sqrt{-\mu}}, & \mu<0.
\end{cases}
\end{equation}

\section{Joint functional limit}
\label{sec:main}

Throughout the remainder of the paper, we assume that
\[
\mu:=\alpha_1\alpha_2<\frac14,
\]
unless explicitly stated otherwise. 

Let $D([0,\infty),\R^2)$ denote the Skorokhod space of c\`adl\`ag $\R^2$-valued functions
on $[0,\infty)$, equipped with the Skorokhod $J_1$-topology.
The following theorem identifies the continuous $\R^2$-valued Gaussian process arising as
the scaling limit of $\{S_n\}_{n\ge0}$.

\begin{theorem}
\label{thm:main}
We have the distributional convergence in $D([0,\infty),\R^2)$,
\begin{equation}
\label{eq:main-convergence}
  \left\{\frac{S_{[nt]}}{\sqrt{n}}\right\}_{t\ge0}
  \ \Longrightarrow\
  \{Z_t\}_{t\ge0}
  \qquad\text{as }n\to\infty,
\end{equation}
where $\{Z_t\}_{t\ge0}$ is a centered continuous $\R^2$-valued Gaussian process with $Z_0=0$,
represented by
\begin{equation}
\label{eq:gaussian-conti}
  Z_t
  =\int_0^t (t/s)^A\,dB_s,
  \qquad t>0,
\end{equation}
for a two-dimensional standard Brownian motion $\{B_t\}_{t\ge0}$.
\end{theorem}

\begin{remark}
\label{rem:explicit-Z} 
(i) In the one-dimensional setting, for $0<\alpha<1/2$, the process $Z_t=\int_0^t(t/s)^\alpha\,dB_s$ is the noise-reinforced Brownian motion, and it is the scaling limit of the
one-dimensional ERW in the reinforcing part of the diffusive regime (cf.\ Baur--Bertoin~\cite{baur2016elephant}, Bertoin~\cite{bertoin2020universality}).

(ii) \Cref{thm:main} is a functional extension of the distributional limit theorem of
Aguech and Qin~\cite{aguech2026two}.

(iii)
Suppose that $\mu=\alpha_1\alpha_2>0$. Set
\[
  P_\gamma:=
  \frac1{\sqrt2}
  \begin{pmatrix}
    \gamma&1\\
    1&-\gamma^{-1}
  \end{pmatrix},
  \qquad
  \gamma:=\sqrt{\frac{\alpha_1}{\alpha_2}},
  \qquad
  \widetilde Z:=P_\gamma^{-1}Z,
  \qquad
  \widetilde B:=P_\gamma^{-1}B.
\]
Writing $\varepsilon:=\operatorname{sgn}(\alpha_1)
  =\operatorname{sgn}(\alpha_2)$, we have
\[
  \widetilde Z_t^{(1)}
  =
  \int_0^t
  \left(\frac ts\right)^{\varepsilon\sqrt{\mu}}
  \,d\widetilde B_s^{(1)},
  \qquad
  \widetilde Z_t^{(2)}
  =
  \int_0^t
  \left(\frac ts\right)^{-\varepsilon\sqrt{\mu}}
  \,d\widetilde B_s^{(2)}.
\]
Moreover, since $P_\gamma$ is symmetric,
$
  d\langle\widetilde B\rangle_t=P_\gamma^{-2}\,dt.
$
Thus, $\widetilde B^{(i)}$ is a Brownian motion with variance
$(P_\gamma^{-2})_{ii}t$, $i=1,2$, and the corresponding stochastic
integral carries this variance factor. In general,
$\widetilde B^{(1)}$ and $\widetilde B^{(2)}$ are not independent.

In particular, if $\alpha_1=\alpha_2=\alpha>0$, then, with
$W:=P_1^\top B$,
the processes $W^{(1)}$ and $W^{(2)}$ are independent standard Brownian
motions, and
\[
  \frac{Z_t^{(1)}+Z_t^{(2)}}{\sqrt2}
  =
  \int_0^t
  \left(\frac ts\right)^\alpha\,dW_s^{(1)},
  \qquad
  \frac{Z_t^{(1)}-Z_t^{(2)}}{\sqrt2}
  =
  \int_0^t
  \left(\frac ts\right)^{-\alpha}\,dW_s^{(2)}.
\]
Thus, in the symmetric case, the sum and difference components are
independent one-dimensional Gaussian processes. In particular, the
difference eliminates the growing mode with parameter $\alpha$, leaving
the mode with $-\alpha$. This suggests that the difference may remain
diffusive beyond the joint diffusive regime. \Cref{thm:diff-fclt}
proves this directly for $1/2\le\alpha<1$, where the above
two-dimensional Wiener integral is no longer defined.
\end{remark}

\begin{proposition}
\label{prop:covariance}
The Wiener integral in~\eqref{eq:gaussian-conti} is well-defined for every $t>0$. Furthermore, the covariance matrix of $\{Z_t\}_{t\ge0}$ is given by
\begin{equation}
\label{eq:cov-ZsZt}
  \E\bigl[Z_s Z_t^\top\bigr]
  =s\,\Sigma\,(t/s)^{A^\top},
  \qquad 0<s\le t,
\end{equation}
where
\begin{equation}
\label{eq:Sigma}
  \Sigma:=\E[Z_1 Z_1^\top]
  =\frac{1}{1-4\alpha_1\alpha_2}
  \begin{pmatrix}
    1+2\alpha_1^2-2\alpha_1\alpha_2
    &\alpha_1+\alpha_2\\[1mm]
    \alpha_1+\alpha_2
    &1+2\alpha_2^2-2\alpha_1\alpha_2
  \end{pmatrix}.
\end{equation}
\end{proposition}

\begin{proof}
Choose $\theta\in(0,1/2)$ with
$\theta>\sqrt{\mu}$ when $\mu>0$. It follows from \eqref{eq:e-matrix-function} that,
for some $C>0$,
\begin{equation}
\label{eq:matrix-exp-growth}
  \|x^A\|\le Cx^\theta,\qquad x\ge1.
\end{equation}
Hence, using $\|M\|_{\HS}\le\sqrt2\|M\|$,
\[
  \int_0^t\|(t/s)^A\|_{\HS}^2\,ds
  \le 2C^2\int_0^t(t/s)^{2\theta}\,ds
  =\frac{2C^2t}{1-2\theta}<\infty.
\]
Thus the Wiener integral defining $Z_t$ is well-defined.

For $0<s\le t$, the It\^o isometry and the semigroup property give
\begin{align*}
  \E[Z_sZ_t^\top]
  =\left(\int_0^s (s/u)^A(s/u)^{A^\top}\,du\right)
    (t/s)^{A^\top}
  =s\,\Sigma\,(t/s)^{A^\top},
\end{align*}
where
$
  \Sigma=\int_0^1 (1/v)^A(1/v)^{A^\top}\,dv.
$
It remains to compute $\Sigma$. With $v=e^{-r}$ and
$B:=A-\tfrac12 I$,
$
  \Sigma=\int_0^\infty e^{rB}e^{rB^\top}\,dr.
$
The preceding bound \eqref{eq:matrix-exp-growth} yields
$
  \|e^{rB}\|=e^{-r/2}\|e^{rA}\|
  \le Ce^{-(1/2-\theta)r},
$
so the integral is absolutely convergent and
$e^{rB}e^{rB^\top}\to0$ as $r\to\infty$. Therefore,
\begin{align*}
  B\Sigma+\Sigma B^\top
  &=\int_0^\infty \frac{d}{dr}
      \bigl(e^{rB}e^{rB^\top}\bigr)\,dr
   =-I,
\end{align*}
which is equivalent to the Lyapunov equation
$
  \Sigma-A\Sigma-\Sigma A^\top=I.
$
Solving this yields \eqref{eq:Sigma}.
\end{proof}

\begin{proposition}
\label{prop:well-defined}
The following statements hold.
\begin{enumerate}
\item The process $Z$ admits a continuous modification.
\item The process $Z$ satisfies the stochastic integral equation
\begin{equation}
\label{eq:limiting_process_volterra}
    Z_t=B_t+
\int_0^t\frac1sAZ_s\,ds,\qquad t\ge0,
\end{equation}
where $\{B_t\}_{t\ge0}$ is a two-dimensional standard Brownian motion.
\end{enumerate}
\end{proposition}

\begin{proof}
Fix $T>0$. By \Cref{prop:covariance},
$
\E\|Z_t\|^2=t\,\operatorname{tr}(\Sigma),
$
$t\ge0$.
For $0<s\le t\le T$, the semigroup property gives
\[
  Z_t-Z_s
  =
  \bigl((t/s)^A-I\bigr)Z_s
  +
  \int_s^t (t/u)^A\,dB_u,
\]
where the two terms on the right-hand side are orthogonal in $L^2$.
If $s\le t\le2s$, then \eqref{eq:matrix-exp-growth},
\[
  (t/s)^A-I
  =
  \int_s^t\frac1rA(r/s)^A\,dr,
\]
and It\^o's isometry yield
\[
  \E\|Z_t-Z_s\|^2\le C(t-s).
\]
If $t>2s$, the same estimate follows from
\[
  \E\|Z_t-Z_s\|^2
  \le2\E\|Z_t\|^2+2\E\|Z_s\|^2
  \le C(t-s),
\]
and the case $s=0$ is immediate. Hence
\[
  \E\|Z_t-Z_s\|^2\le C_T|t-s|,
  \qquad 0\le s,t\le T.
\]
Since $Z_t-Z_s$ is a centered Gaussian,
\[
  \E\|Z_t-Z_s\|^4
  \le
  3\bigl(\E\|Z_t-Z_s\|^2\bigr)^2
  \le
  C_T'|t-s|^2.
\]
The Kolmogorov continuity theorem therefore shows that $Z$ admits a
continuous modification.

We next prove \eqref{eq:limiting_process_volterra}. By
\eqref{eq:matrix-exp-growth},
\[
  \int_0^t
  \left(
    \int_0^r
    \left\|
      \frac1rA(r/s)^A
    \right\|_{\HS}^2\,ds
  \right)^{1/2}dr
  \le C\int_0^t r^{-1/2}\,dr<\infty.
\]
Thus the stochastic Fubini theorem
(see, e.g., Protter~\cite[Theorem~IV.65]{philip2005stochastic}) yields,
for $0<\varepsilon<t$,
\begin{equation}
\label{eq:epsilon-volterra}
  Z_t-Z_\varepsilon
  =B_t-B_\varepsilon
  +\int_\varepsilon^t\frac1rAZ_r\,dr.
\end{equation}
Indeed, this follows by writing
\[
  (t/s)^A-(\varepsilon/s)^A
  =\int_\varepsilon^t\frac1rA(r/s)^A\,dr,
  \qquad s<\varepsilon,
\]
and
\[
  (t/s)^A-I
  =\int_s^t\frac1rA(r/s)^A\,dr,
  \qquad \varepsilon\le s<t,
\]
and then interchanging the order of integration.
Since $\E\|Z_t\|^2=t\,\operatorname{tr}(\Sigma)$,
\[
  \E\left[\int_0^t\frac{\|Z_r\|}{r}\,dr\right]
  \le C\int_0^t r^{-1/2}\,dr<\infty.
\]
Hence $\int_0^t r^{-1}\|Z_r\|\,dr<\infty$ almost surely. Letting
$\varepsilon\downarrow0$ in \eqref{eq:epsilon-volterra} and using the
continuity of $Z$ and $B$, with $Z_0=B_0=0$, gives \eqref{eq:limiting_process_volterra}.
\end{proof}

\begin{remark}
The stochastic integral equation \eqref{eq:limiting_process_volterra}
admits a simple interpretation under the logarithmic time change.
Formally writing
\[
dZ_t=dB_t+\frac1tAZ_t\,dt
\]
and setting $X_s:=e^{-s/2}Z_{e^s}$, $s\in \R$, define
$
W_s
:=
\int_1^{e^s}
u^{-1/2}\,dB_u,
$
$s\in\R$,
where the stochastic integral is understood as an oriented integral when $s<0$. Then $\{W_s\}_{s\in\R}$ is a two-sided,
two-dimensional Brownian motion. A direct application of Itô's formula yields
\[
dX_s
=
\left(
A-\frac12I
\right)
X_s\,ds
+
dW_s.
\]
Thus, the Lamperti transform of $Z$ is an Ornstein--Uhlenbeck process with
drift matrix $A-\frac12I$.
Since the eigenvalues of $A-\frac12I$ have negative real parts under the
assumption $\mu<1/4$,
this equation admits the stationary solution
\[
X_s
=
\int_{-\infty}^s
e^{(A-\frac12I)(s-r)}\,dW_r.
\]
For $s\le t$, we have $\E[X_sX_t^\top]=\Sigma e^{(A^\top-\frac12I)(t-s)}$. In particular, $\E[X_sX_s^\top]=\Sigma$.
\end{remark}

\section{Proof of the joint functional limit}
\label{sec:proof}

\subsection{Martingale functional limit}
\label{subsec:martingale-fclt}

Since
\[
  \E\bigl[X_{n+1}^{(1)}\bigm|\F_n\bigr]=\alpha_1\frac{S_n^{(2)}}{n},
  \qquad
  \E\bigl[X_{n+1}^{(2)}\bigm|\F_n\bigr]=\alpha_2\frac{S_n^{(1)}}{n},
\]
writing $\Delta S_{n+1}:=S_{n+1}-S_n$ yields
\begin{align*}
  \E[\Delta S_{n+1}\mid\F_n]
  &=\E\Biggl[\begin{pmatrix}X_{n+1}^{(1)}\\ X_{n+1}^{(2)}\end{pmatrix}\Biggm|\F_n\Biggr]
  =\begin{pmatrix}
    \alpha_1 S_n^{(2)}/n\\
    \alpha_2 S_n^{(1)}/n
  \end{pmatrix}
  =\frac1n A S_n.
\end{align*}
Therefore the sequence $\{D_{n}\}_{n\ge2}$ defined by
\[
  D_{n+1}:=\Delta S_{n+1}-\frac1n A S_n,\qquad n\ge1,
\]
is a martingale-difference sequence with respect to $\{\F_n\}_{n\ge0}$.
We also set $M_0:=0$. Adopting the convention that an empty sum equals zero, define
\[
  M_n:=\sum_{k=1}^{n-1}D_{k+1},\qquad n\ge1.
\]
Then $\{M_n\}_{n\ge0}$ is a two-dimensional martingale with respect to $\{\F_n\}_{n\ge0}$.
Moreover, for every $n\ge2$,
\begin{align*}
  M_n
  &=\sum_{k=1}^{n-1}D_{k+1}
  =\sum_{k=1}^{n-1}\Bigl(\Delta S_{k+1}-\frac1k A S_k\Bigr)
  =S_n-S_1-\sum_{k=1}^{n-1}\frac1k A S_k.
\end{align*}
Hence
\begin{equation}
\label{eq:ERWs-martingale}
  S_n=S_1+M_n+\sum_{k=1}^{n-1}\frac1k A S_k,\qquad n\ge2.
\end{equation}
This may be viewed as the discrete counterpart of
the stochastic integral equation
\eqref{eq:limiting_process_volterra}
satisfied by the limiting Gaussian process.

The following second-moment estimate for $S_n^{(i)}$ will be used in the proof of \Cref{prop:martingale-fclt}.

\begin{lemma}
\label{lem:square-order}
For $i=1,2$, $ \E\bigl[(S_n^{(i)})^2\bigr]=O(n)$ as $n\to \infty$.
\end{lemma}
\begin{proof}
We consider the three possible signs of $\mu$.
Suppose first that $0<\mu<1/4$. Let
$
  v_\pm^\top=(\alpha_2,\pm\sqrt\mu)
$
and
$
  Y_n^\pm:=v_\pm^\top S_n.
$
Since $v_\pm^\top A=\pm\sqrt\mu\,v_\pm^\top$ and the increments of
$S_n$ are uniformly bounded,
\[
  \E[(Y_{n+1}^\pm)^2]
  \le
  \left(1\pm\frac{2\sqrt\mu}{n}\right)
  \E[(Y_n^\pm)^2]+C.
\]
Set $\beta:=2\sqrt\mu<1$ and
\[
  P_n:=\prod_{k=1}^{n-1}\left(1+\frac{\beta}{k}\right).
\]
Then $P_n\asymp n^\beta$. Writing
$a_n:=\E[(Y_n^+)^2]$ and iterating the above recursion, we obtain
\[
  a_n
  \le
  P_n\left(
    a_1+C\sum_{k=1}^{n-1}P_{k+1}^{-1}
  \right)
  \lesssim
  n^\beta\left(1+\sum_{k=1}^{n-1}k^{-\beta}\right)
  =O(n),
\]
where the last estimate uses $\beta<1$. The same bound holds for
$\E[(Y_n^-)^2]$, since the factor $1-\beta/n$ is no larger than
$1+\beta/n$. Hence
$
  \E[(Y_n^\pm)^2]=O(n).
$
Since $v_+$ and $v_-$ are linearly independent, the same bound holds
for both coordinates of $S_n$.

Suppose next that $\mu<0$. Set
$
  P:=\operatorname{diag}(|\alpha_2|,|\alpha_1|).
$
Since $\alpha_1$ and $\alpha_2$ have opposite signs, $PA+A^\top P=0$. For
$Q_n:=S_n^\top P S_n$,
\begin{align*}
  \E[Q_{n+1}\mid\F_n]
  =Q_n+\frac{2}{n}S_n^\top PA S_n
    +\E[\Delta S_{n+1}^\top P\Delta S_{n+1}\mid\F_n]
  \le Q_n+C,
\end{align*}
where the middle term vanishes because
$2S_n^\top PA S_n=S_n^\top(PA+A^\top P)S_n=0$.
Hence $\E[Q_n]=O(n)$. Since
$
  Q_n=|\alpha_2|(S_n^{(1)})^2+|\alpha_1|(S_n^{(2)})^2,
$
and $\alpha_1\alpha_2\ne0$, both coordinate bounds follow.

Finally, suppose that $\mu=0$. By symmetry, assume $\alpha_2=0$. Then
$\E[(S_n^{(2)})^2]=n$. Put
$
  c_n:=\E[S_n^{(1)}S_n^{(2)}]
$
and
$
  d_n:=\E[(S_n^{(1)})^2].
$
Using
$
  \E[X_{n+1}^{(1)}\mid\F_n]=\frac{\alpha_1}{n}S_n^{(2)}
$
and
$
  \E[X_{n+1}^{(2)}\mid\F_n]=0,
$
together with
$\E[X_{n+1}^{(1)}X_{n+1}^{(2)}\mid\F_n]=0$, we obtain
\[
  c_{n+1}=c_n+\alpha_1,
  \qquad
  d_{n+1}=d_n+\frac{2\alpha_1}{n}c_n+1.
\]
Hence $c_n=O(n)$ and then $d_n=O(n)$. This completes the proof.
\end{proof}

\begin{proposition}
\label{prop:martingale-fclt}
We have the distributional convergence in $D([0,\infty),\R^2)$,
\begin{equation}
\label{eq:martingale-fclt}
  \Bigl\{\frac{M_{[nt]}}{\sqrt{n}}\Bigr\}_{t\ge0}
  \ \Longrightarrow\
  \{B_t\}_{t\ge0}
  \qquad\text{as }n\to\infty,
\end{equation}
where $\{B_t\}_{t\ge0}$ is a two-dimensional standard Brownian motion.
\end{proposition}
\begin{proof}
Set
\[
  M_n(t):=\frac{M_{[nt]}}{\sqrt n},
  \qquad
  \mathcal G_t^{(n)}:=\mathcal F_{[nt]}.
\]
We verify the conditions of the martingale functional central limit theorem
\cite[Theorem~2.1(ii)]{whitt2007proofs} (see also
Ethier--Kurtz~\cite{ethier1986markov},
Hall--Heyde~\cite{hall1980martingale} and
Jacod--Shiryaev~\cite{jacod2003limit}).

Since $\|D_{k+1}\|\le2\sqrt2$, the maximum jump of $M_n$ is bounded by
$2\sqrt{2}/\sqrt{n}$ and hence tends to zero. Moreover, for every
$\varepsilon>0$, the indicator
$
  \mathbf{1}_{\{\|D_{k+1}\|>\varepsilon\sqrt n\}}
$
vanishes for all sufficiently large $n$, so the Lindeberg condition is
satisfied.

For the predictable quadratic covariations, writing
$\nu_n(t):=([nt]-1)_+$, we have
\[
  \langle M_n^{(i)},M_n^{(j)}\rangle(t)
  =
  \frac1n\sum_{k=1}^{\nu_n(t)}
  \E[D_{k+1}^{(i)}D_{k+1}^{(j)}\mid\F_k].
\]
In particular,
\[
  \langle M_n^{(i)}\rangle(t)
  =
  \frac{\nu_n(t)}{n}
  -
  \frac{\alpha_i^2}{n}
  \sum_{k=1}^{\nu_n(t)}
  \left(\frac{S_k^{(3-i)}}{k}\right)^2.
\]
By~\Cref{lem:square-order}, for every $T>0$,
\[
  \E\left[
    \sup_{0\le t\le T}
    \frac1n
    \sum_{k=1}^{\nu_n(t)}
    \left(\frac{S_k^{(3-i)}}{k}\right)^2
  \right]
  \le
  \frac{C}{n}\sum_{k=1}^{\nu_n(T)}\frac1k
  \longrightarrow0.
\]
Since $\sup_{0\le t\le T}|\nu_n(t)/n-t|\le2/n$, it follows from
Markov's inequality that
\[
  \sup_{0\le t\le T}
  |\langle M_n^{(i)}\rangle(t)-t|
  \xrightarrow{\Prob}0.
\]
Moreover,
\[
  \E[X_{k+1}^{(1)}X_{k+1}^{(2)}\mid\F_k]
  =
  \alpha_1\alpha_2
  \frac{S_k^{(1)}S_k^{(2)}}{k^2}.
\]
Since
\[
  \E[X_{k+1}^{(1)}\mid\F_k]
  =
  \alpha_1\frac{S_k^{(2)}}{k},
  \qquad
  \E[X_{k+1}^{(2)}\mid\F_k]
  =
  \alpha_2\frac{S_k^{(1)}}{k},
\]
we obtain
\[
  \E[D_{k+1}^{(1)}D_{k+1}^{(2)}\mid\F_k]
  =
  \E[X_{k+1}^{(1)}X_{k+1}^{(2)}\mid\F_k]
  -
  \E[X_{k+1}^{(1)}\mid\F_k]
  \E[X_{k+1}^{(2)}\mid\F_k]
  =
  0.
\]
Hence $\langle M_n^{(1)},M_n^{(2)}\rangle\equiv0$. Thus, for
$i,j\in\{1,2\}$,
\[
  \sup_{0\le t\le T}
  \bigl|
    \langle M_n^{(i)},M_n^{(j)}\rangle(t)-\delta_{ij}t
  \bigr|
  \xrightarrow{\Prob}0.
\]
Finally, since the jumps of the predictable covariations are bounded by
$4/n$, their maximum jumps vanish. Whitt's theorem now yields
\eqref{eq:martingale-fclt}.
\end{proof}

\subsection{Functional limit for the rescaled walks}
\label{subsec:aux}

We retain the notation of \Cref{subsec:martingale-fclt}. In particular,
\[
 A=\begin{pmatrix}0&\alpha_1\\\alpha_2&0\end{pmatrix},\quad
 \mu=\alpha_1\alpha_2<\frac{1}{4},\quad
 D_{k+1}=S_{k+1}-S_k-\frac{AS_k}{k},\quad
 M_m=\sum_{k=1}^{m-1}D_{k+1}.
\]
Fix $0<\theta<1/2$, with $\theta>\sqrt{\mu}$ when $\mu>0$. Throughout this subsection, constants may depend on the fixed parameters. We use the bound $\|x^A\|\le Cx^\theta$ for $x\ge1$, established in \eqref{eq:matrix-exp-growth}. The proof consists of convergence away from the origin, a maximal estimate controlling the omitted part, and removal of the truncation.

\begin{lemma}
\label{lem:fundamental}
For integers $1\le\ell\le m$, let
\[
 \Phi(m,\ell)=\prod_{j=\ell}^{m-1}(I+A/j),\qquad \Phi(m,m)=I.
\]
Then the following assertions hold.
\begin{enumerate}
\item For $m\ge1$,
\begin{equation}\label{sec43:decomposition}
 S_m=\Phi(m,1)S_1+Y_m,\qquad
 Y_m=\sum_{k=1}^{m-1}\Phi(m,k+1)D_{k+1},\qquad Y_0=0.
\end{equation}
This identity also holds without the assumption $\mu<1/4$.
\item For every $0<\delta<T<\infty$,
\[
 \sup_{\delta\le s\le t\le T}
 \|\Phi(\lfloor nt\rfloor,\lfloor ns\rfloor)-(t/s)^A\|\longrightarrow0.
\]
\item For all $1\le\ell\le m$, $\|\Phi(m,\ell)\|\le C(m/\ell)^\theta$.
\end{enumerate}
\end{lemma}

\begin{proof}
The recursion $S_{k+1}=(I+A/k)S_k+D_{k+1}$ gives (1) by iteration. For sufficiently large $j$,
\[
 \log(I+A/j)=A/j+O(j^{-2}).
\]
All matrices in these expansions commute with $A$. Since
$\sum_{j=\ell}^{m-1}j^{-1}=\log(m/\ell)+O(\ell^{-1})$, uniformly for $m\ge\ell$, we obtain
\begin{equation}\label{sec43:factorization}
 \Phi(m,\ell)=(m/\ell)^Ae^{E_{m,\ell}},\qquad
 \|E_{m,\ell}\|\le C/\ell
\end{equation}
for all sufficiently large $\ell$. Hence
\begin{equation}\label{sec43:kernel-error}
 \|\Phi(m,\ell)-(m/\ell)^A\|
 \le\frac{C}{\ell}(m/\ell)^\theta.
\end{equation}
On $\delta\le s\le t\le T$, the ratio $\lfloor nt\rfloor/\lfloor ns\rfloor$ equals $t/s+O(n^{-1})$ uniformly. The map $x\mapsto x^A$ has bounded derivative on the resulting compact interval, so \eqref{sec43:kernel-error} proves (2), in fact with error $O(n^{-1})$. Formula \eqref{sec43:factorization} also proves (3) for large $\ell$. The remaining starting indices are absorbed using $\Phi(m,\ell)=\Phi(m,j_0)\Phi(j_0,\ell)$ and enlarging the constant over the finitely many pairs $\ell\le m<j_0$.
\end{proof}

Fix $0<\varepsilon<T$. For $n$ sufficiently large that $a:=\lfloor n\varepsilon\rfloor\ge1$, and $m:=\lfloor nt\rfloor$, define, for $x\in D([0,T],\R^2)$,
\[
 G_n^{(\varepsilon)}(x)(t)=
 \begin{cases}
 \displaystyle\sum_{k=a}^{m-1}\Phi(m,k+1)
 \{x((k+1)/n)-x(k/n)\},&m>a,\\[3pt]
 0,&m=a,
 \end{cases}
 \qquad \varepsilon\le t\le T.
\]
The limiting map is
\begin{equation}\label{sec43:limit-map}
 G^{(\varepsilon)}(x)(t)=x(t)-(t/\varepsilon)^Ax(\varepsilon)
 +\int_\varepsilon^t(t/s)^A\frac{Ax(s)}{s}ds,
 \qquad \varepsilon\le t\le T.
\end{equation}

\begin{lemma}
\label{lem:truncated-transform-approx}
If $x_n\to x$ in the Skorokhod $J_1$ topology on $D([0,T],\R^2)$ and $x$ is continuous, then $G_n^{(\varepsilon)}(x_n)\to G^{(\varepsilon)}(x)$ uniformly on $[\varepsilon,T]$.
\end{lemma}
\begin{proof}
Since, on a compact interval, $J_1$ convergence to a continuous limit implies uniform convergence, the assumption gives $\|x_n-x\|_\infty\to0$. For $m>a$, summation by parts and the identity $\Phi(m,k)-\Phi(m,k+1)=\Phi(m,k+1)A/k$ yield
\begin{equation}\label{sec43:abel}
 G_n^{(\varepsilon)}(x_n)(t)
 =x_n(m/n)-\Phi(m,a+1)x_n(a/n)
 +\sum_{k=a+1}^{m-1}\frac{1}{k}\Phi(m,k+1)Ax_n(k/n).
\end{equation}
The kernels converge uniformly to their continuous counterparts by \eqref{sec43:kernel-error}. Their norms and the sums of the coefficient norms in \eqref{sec43:abel} are bounded by a constant depending only on $\varepsilon,T,A$. Replacing $x_n$ by $x$ therefore gives a uniform error at most $C_{\varepsilon,T}\|x_n-x\|_\infty$. The remaining sum is a Riemann sum for the integral in \eqref{sec43:limit-map}, uniformly in its upper endpoint, because $(t,s)\mapsto(t/s)^AAx(s)/s$ is continuous on $\varepsilon\le s\le t\le T$.

More explicitly, with $\omega_x(h)=\sup_{|u-v|\le h}\|x(u)-x(v)\|$ on $[0,T]$, these estimates give
\[
 \sup_{\varepsilon\le t\le T}\|G_n^{(\varepsilon)}(x_n)(t)-G^{(\varepsilon)}(x)(t)\|
 \le C_{\varepsilon,T}\left(\|x_n-x\|_\infty+
       \omega_x(2/n)+\frac{\|x\|_\infty}{n}\right).
\]
The initial interval $\varepsilon\le t<(a+1)/n$ satisfies the same bound, since the discrete map is zero there and $G^{(\varepsilon)}(x)(\varepsilon)=0$. The right-hand side tends to zero.
\end{proof}

\begin{proposition}
\label{prop:truncated}
Fix $0<\varepsilon<T$. For all sufficiently large $n$, put
\[
 Y_n^{(\varepsilon)}(t)=\frac{1}{\sqrt{n}}
 \sum_{k=\lfloor n\varepsilon\rfloor}^{\lfloor nt\rfloor-1}
 \Phi(\lfloor nt\rfloor,k+1)D_{k+1},\qquad\varepsilon\le t\le T,
\]
where an empty sum is zero. Then
\[
 Y_n^{(\varepsilon)}\Longrightarrow Z^{(\varepsilon)}
 \quad\text{in }D([\varepsilon,T],\R^2),\qquad
 Z_t^{(\varepsilon)}=\int_\varepsilon^t(t/s)^A dB_s.
\]
\end{proposition}

\begin{proof}
Let $\mathcal M_n(t)=M_{\lfloor nt\rfloor}/\sqrt{n}$. \Cref{prop:martingale-fclt} gives $\mathcal M_n\Longrightarrow B$, and $Y_n^{(\varepsilon)}=G_n^{(\varepsilon)}(\mathcal M_n)$. \Cref{lem:truncated-transform-approx} and the extended continuous mapping theorem \cite[Theorem~1.11.1]{wellner2023weak} imply convergence to $G^{(\varepsilon)}(B)$. Finally, $\partial_s(t/s)^A=-(t/s)^AA/s$, so integration by parts identifies \eqref{sec43:limit-map} with $\int_\varepsilon^t(t/s)^A dB_s$.
\end{proof}

\begin{lemma}
\label{lem:fundamental-maximal}
Let $(\xi_{k+1})_{k\ge1}$ be an $\R^2$-valued martingale-difference sequence with $\sup_k\E\|\xi_{k+1}\|^2<\infty$. For
\[
 U_m=\sum_{k=1}^{m-1}\Phi(m,k+1)\xi_{k+1},\qquad U_0=U_1=0,
\]
there is a constant $C$ such that $\E\max_{0\le m\le N}\|U_m\|^2\le CN$ for all $N\ge1$.
\end{lemma}

\begin{proof}
Orthogonality of the martingale differences and \Cref{lem:fundamental}(3) give
\begin{equation}\label{sec43:pointwise-second}
 \E\|U_m\|^2\le Cm^{2\theta}
 \sum_{k=1}^{m-1}(k+1)^{-2\theta}\le C'm.
\end{equation}
Since $A^2=\mu I$ and $\mu<1/4$,
\[
 (I+A/j)^{-1}=\frac{I-A/j}{1-\mu/j^2},\qquad
 \|(I+A/j)^{-1}\|\le1+C/j.
\]
This bound holds for all three signs of $\mu$. Together with \Cref{lem:fundamental}(3), it implies
\begin{equation}\label{sec43:dyadic}
 \sup_{r\ge0}\sup_{2^r\le m\le2^{r+1}}
 \{\|\Phi(m,2^r)\|+\|\Phi(m,2^r)^{-1}\|\}<\infty,
\end{equation}
because the harmonic sum over each dyadic block is bounded. For $a=2^r$ and $a\le m\le2a$, factor
\[
 U_m=\Phi(m,a)\left\{U_a+
       \sum_{k=a}^{m-1}\Phi(k+1,a)^{-1}\xi_{k+1}\right\}.
\]
The sum in braces after $U_a$ is a martingale in $m$. Doob's $L^2$ inequality, \eqref{sec43:pointwise-second}, and \eqref{sec43:dyadic} give
\[
 \E\max_{a\le m\le2a}\|U_m\|^2
 \le C\E\|U_a\|^2+C\sum_{k=a}^{2a-1}\E\|\xi_{k+1}\|^2
 \le C'a.
\]
Summing these bounds over $a=1,2,4,\ldots,2^{\lfloor\log_2N\rfloor}$ proves the result.
\end{proof}

\begin{proposition}
\label{prop:origin}
Fix $T>0$, set $Y_n(t)=Y_{\lfloor nt\rfloor}/\sqrt{n}$, and extend $Y_n^{(\varepsilon)}$ by zero on $[0,\varepsilon)$. Then, for every $\eta>0$,
\[
 \lim_{\varepsilon\downarrow0}\limsup_{n\to\infty}
 \Prob\left(\sup_{0\le t\le T}\|Y_n(t)-Y_n^{(\varepsilon)}(t)\|>\eta\right)=0.
\]
\end{proposition}

\begin{proof}
Write $a=\lfloor n\varepsilon\rfloor\ge1$. The martingale differences $D_{k+1}$ are bounded, so \Cref{lem:fundamental-maximal} gives
$\E\max_{0\le m\le N}\|Y_m\|^2\le CN$.
For $t<\varepsilon$, the expected squared maximum of the error is therefore at most $Ca/n$. For $t\ge\varepsilon$, the multiplicative property of $\Phi$ gives the exact identity
\[
 Y_n(t)-Y_n^{(\varepsilon)}(t)
 =\frac{1}{\sqrt{n}}\Phi(\lfloor nt\rfloor,a)Y_a.
\]
Consequently, for $0<\varepsilon<\min\{1,T\}$ and sufficiently large $n$,
\begin{align}
  \E\sup_{0\le t\le T}
  \|Y_n(t)-Y_n^{(\varepsilon)}(t)\|^2
  &\le
  \frac{Ca}{n}
  +
  \frac{C}{n}
  \left(\frac{nT}{a}\right)^{2\theta}
  \E\|Y_a\|^2
  \notag\\
  &\le
  \frac{Ca}{n}
  +
  CT^{2\theta}n^{2\theta-1}a^{1-2\theta}
  \notag\\
  &\le
  C_T\{\varepsilon+\varepsilon^{1-2\theta}\}.
  \label{sec43:origin-bound}
\end{align}
Here we used $\E\|Y_a\|^2\le Ca$ and
$a=\lfloor n\varepsilon\rfloor\asymp n\varepsilon$. Since
$\theta<1/2$, both terms on the right-hand side of
\eqref{sec43:origin-bound} tend to zero as $\varepsilon\downarrow0$.
Markov's inequality proves the assertion.
\end{proof}

\begin{proposition}
\label{prop:discrete_sum_convergence_Z}
For every $T>0$, $Y_n\Longrightarrow Z$ in $D([0,T],\R^2)$.
\end{proposition}

\begin{proof}
Extend $Z^{(\varepsilon)}$ by zero on $[0,\varepsilon)$. It is continuous at $\varepsilon$, so \Cref{prop:truncated} also gives $Y_n^{(\varepsilon)}\Longrightarrow Z^{(\varepsilon)}$ on $[0,T]$. For $t\ge\varepsilon$,
\[
 Z_t-Z_t^{(\varepsilon)}=(t/\varepsilon)^AZ_\varepsilon.
\]
The kernel bound and $\E\|Z_\varepsilon\|^2=\varepsilon\operatorname{tr}\Sigma$ yield
\begin{equation}\label{sec43:limit-origin}
 \E\sup_{\varepsilon\le t\le T}\|Z_t-Z_t^{(\varepsilon)}\|^2
 \le C_T\varepsilon^{1-2\theta}\longrightarrow0.
\end{equation}
On $[0,\varepsilon)$ the error tends to zero almost surely by continuity of $Z$ at zero (\Cref{prop:well-defined}). Thus $Z^{(\varepsilon)}\to Z$ uniformly in probability. This, \Cref{prop:origin}, and the converging-together theorem \cite[Theorem~3.1]{billingsley1999convergence} prove the result. The uniform error bounds apply to the $J_1$ topology because its metric can be bounded above by the uniform metric.
\end{proof}

\begin{proof}[Proof of \Cref{thm:main}]
For $t\ge1/n$, \eqref{sec43:decomposition} gives
$S_{\lfloor nt\rfloor}/\sqrt{n}=\Phi(\lfloor nt\rfloor,1)S_1/\sqrt{n}+Y_n(t)$;
both $S_{\lfloor nt\rfloor}$ and $Y_n(t)$ vanish for $t<1/n$. Hence, for every fixed $T>0$,
\[
 \sup_{0\le t\le T}\left\|\frac{S_{\lfloor nt\rfloor}}{\sqrt{n}}-Y_n(t)\right\|
 \le C_T n^{\theta-1/2}\|S_1\|\longrightarrow0
 \quad\text{a.s.}
\]
\Cref{prop:discrete_sum_convergence_Z} and the asymptotic-equivalence theorem \cite[Theorem~3.1]{billingsley1999convergence} give the claimed convergence on $[0,T]$. Since $T$ is arbitrary and the limit is continuous, the convergence holds in $D([0,\infty),\R^2)$.
\end{proof}

\section{Difference process}
\label{sec:difference}

In this section, we study the difference between the two interacting elephant random walks.
Define
\begin{equation}
\label{eq:diff-def}
  \mathcal{D}_n:=S_n^{(1)}-S_n^{(2)},\qquad n\ge0.
\end{equation}
Writing $\dd:=(1,-1)^\top$, we have $\mathcal{D}_n=\dd^\top S_n$.

Throughout the remainder of this section, unless otherwise stated, we assume either
\begin{equation*}
    \mu<1/4\qquad\text{or}\qquad \mu\ge 1/4\quad \text{with}\quad \alpha_1=\alpha_2\in[1/2,1).
\end{equation*}

\subsection{Functional limit}

The following theorem shows that, in the symmetric case
$\alpha_1=\alpha_2\in[1/2,1)$, the same $\sqrt n$-scaling still yields
a nontrivial functional limit, despite the fact that the joint process is critical or superdiffusive. 

\begin{theorem}
\label{thm:diff-fclt}
We have the distributional convergence in $D([0,\infty),\R)$,
\begin{equation}
\label{eq:diff-fclt}
  \Bigl\{\frac{\mathcal{D}_{[nt]}}{\sqrt{n}}\Bigr\}_{t\ge0}
  \ \Longrightarrow\
  \{V_t\}_{t\ge0}
  \qquad\text{as }n\to\infty,
\end{equation}
where $\{V_t\}_{t\ge0}$ is a centered continuous real-valued Gaussian process with $V_0=0$, represented as follows.

\begin{enumerate}
\item If $\mu<1/4$, then
\begin{equation}
\label{eq:diff-limit}
  V_t=\dd^\top Z_t
  =\int_0^t \dd^\top(t/s)^A\,dB_s,\qquad t>0,
\end{equation}
where $\{B_t\}_{t\ge 0}$ is a two-dimensional standard Brownian motion.

\item If $\alpha_1=\alpha_2\in[1/2,1)$, then 
\begin{equation} 
\label{eq:diff-limit-symmetric} 
V_t = \sqrt{2}\int_0^t \left(\frac{t}{s}\right)^{-\alpha_1}\,dW_s, \qquad t>0, 
\end{equation} 
where $\{W_t\}_{t\ge0}$ is a standard Brownian motion.
\end{enumerate}
\end{theorem}

\begin{corollary}
\label{cor:diff-volterra}
    The limiting process satisfies the following stochastic equations.
\begin{enumerate}
\item If $\mu<1/4$, then
    \begin{equation}
\label{eq:diff-volterra}
    V_t=
\dd^\top B_t
+
\int_0^t
\dd^\top A Z_s\,\frac{ds}{s},
\qquad t\ge0,
\end{equation}
where $\{\dd^\top B_t\}_{t\ge 0}$ is a one-dimensional Brownian motion with variance parameter $2$.

\item If $\alpha_1=\alpha_2\in[1/2,1)$, then
\begin{equation}
\label{eq:diff-volterra-symmetric}
V_t
=
\sqrt{2}\,W_t
-
\alpha_1\int_0^t\frac{V_s}{s}\,ds,
\qquad t\ge0,
\end{equation}
where $\{W_t\}_{t\ge0}$ is a standard Brownian motion.
\end{enumerate}
\end{corollary}
\begin{proof}
If $\mu<1/4$, \eqref{eq:diff-volterra} follows immediately from
\eqref{eq:limiting_process_volterra} and \eqref{eq:diff-limit}.

Suppose that $\alpha_1=\alpha_2\in[1/2,1)$. By
\eqref{eq:diff-limit-symmetric},
\[
  V_t
  =
  \sqrt{2}\,t^{-\alpha_1}
  \int_0^t s^{\alpha_1}\,dW_s.
\]
Integration by parts on $[\varepsilon,t]$ gives
\[
  V_t-V_\varepsilon
  =
  \sqrt{2}\,(W_t-W_\varepsilon)
  -
  \alpha_1\int_\varepsilon^t\frac{V_s}{s}\,ds.
\]
Moreover, It\^o's isometry yields
$
  \E[V_s^2]=\frac{2s}{1+2\alpha_1},
$
and hence
\[
  \E\left[\int_0^t\frac{|V_s|}{s}\,ds\right]
  \le C\int_0^t s^{-1/2}\,ds<\infty.
\]
Thus the integral is almost surely finite, and letting
$\varepsilon\downarrow0$ proves \eqref{eq:diff-volterra-symmetric}.
\end{proof}

For $\mu<1/4$, define
\begin{equation}
\label{eq:key_parameter_diff_covariance}
\lambda
:=
\frac{\nu(1-\nu)}{1-4\mu},
\qquad
\nu:=\alpha_1+\alpha_2.
\end{equation}

\begin{proposition}
\label{prop:diff-cov}
The covariance kernel of $\{V_t\}_{t\ge0}$ is given as follows.
\begin{enumerate}
\item
If $\mu<1/4$, then
\[
\E\bigl[V_sV_t\bigr]
=
2r\,
c_\mu\!\left(
\log\frac{q}{r}
\right)
-
\lambda\,C_\mu(s,t),\qquad C_\mu(s,t)
:=
r
\left(
2c_\mu\!\left(
\log\frac{q}{r}
\right)
+
s_\mu\!\left(
\log\frac{q}{r}
\right)
\right),
\]
where $r:=\min\{s,t\}$, $q:=\max\{s,t\}$, and
$c_\mu$ and $s_\mu$ are defined in~\eqref{function:convenient_e-matrix}.

\item
If $\alpha_1=\alpha_2\in(-1/2,1)$, then
\begin{equation}
\label{eq:diff-cov-symmetric}
\E[V_sV_t]
=
\frac{2s}{1+2\alpha_1}
\left(\frac{t}{s}\right)^{-\alpha_1},
\qquad 0<s\le t.
\end{equation}
\end{enumerate}

In particular,
\begin{equation}
\label{eq:diff-variance}
\E[V_t^2]
=
\begin{cases}
\displaystyle
\frac{2t\bigl\{1-\alpha_1-\alpha_2+(\alpha_1-\alpha_2)^2\bigr\}}
     {1-4\alpha_1\alpha_2},
& \mu<1/4,\\[3mm]
\displaystyle
\frac{2t}{1+2\alpha_1},
& \alpha_1=\alpha_2\in[1/2,1).
\end{cases}
\end{equation}
\end{proposition}

\begin{proof}
Suppose first that $\mu<1/4$.
Since $V_t=\dd^\top Z_t$, \Cref{prop:covariance} gives
\[
\E[V_sV_t]
=
r\,\dd^\top\Sigma
(q/r)^{A^\top}\dd,
\]
where $r=\min\{s,t\}$ and $q=\max\{s,t\}$.
By~\eqref{eq:e-matrix-function},
\[
e^{xA^\top}
=
c_\mu(x)I+s_\mu(x)A^\top.
\]
Moreover, a direct calculation yields
$
\dd^\top\Sigma \dd=2(1-\lambda)
$
and
$
\dd^\top\Sigma A^\top \dd=-\lambda.
$
Substituting these identities into the above representation yields the first assertion.

Suppose now that
$\alpha_1=\alpha_2=\alpha\in(-1/2,1/2)$.
If $\alpha\ne0$, then
\[
  \mu=\alpha^2,
  \qquad
  \nu=2\alpha,
  \qquad
  \lambda
  =
  \frac{\nu(1-\nu)}{1-4\mu}
  =
  \frac{2\alpha}{1+2\alpha}.
\]
For $0<s\le t$, set $h:=\log(t/s)$. By the first assertion,
\begin{align*}
  \E[V_sV_t]
   =
  \frac{2s}{1+2\alpha}e^{-\alpha h}
  =
  \frac{2s}{1+2\alpha}
  \left(\frac{t}{s}\right)^{-\alpha}.
\end{align*}
For $\alpha=0$, both sides are equal to $2s$.
Thus, \eqref{eq:diff-cov-symmetric} holds for
$\alpha\in(-1/2,1/2)$.

For $\alpha_1=\alpha_2\in[1/2,1)$,
\eqref{eq:diff-limit-symmetric} and It\^o's isometry yield
\eqref{eq:diff-cov-symmetric}.
Hence, \eqref{eq:diff-cov-symmetric} holds for every
$\alpha_1=\alpha_2\in(-1/2,1)$.

Finally, when $\mu<1/4$, taking $s=t$ and using
$c_\mu(0)=1$ and $s_\mu(0)=0$ gives
\[
\E[V_t^2]
=
2t(1-\lambda)
=
\frac{2t\bigl\{1-\alpha_1-\alpha_2+(\alpha_1-\alpha_2)^2\bigr\}}
     {1-4\alpha_1\alpha_2}.
\]
Together with~\eqref{eq:diff-cov-symmetric}, this proves
\eqref{eq:diff-variance}.
\end{proof}

\begin{remark}
(i) When $\alpha_2=0$, that is when the second walk is a simple symmetric random walk except for its initial step, $\E\bigl[V_sV_t\bigr]
  =2\bigl(1-\alpha_1(1-\alpha_1)\bigr)s-\alpha_1(1-\alpha_1)\,s\log(t/s)$ for $0<s\le t$.
Furthermore:
\begin{enumerate}
\item If $\alpha_1\in\{0,1\}$, then
$\{V_t\}_{t\ge 0}$ has the law of
$\{\sqrt2\,B_t\}_{t\ge 0}$ for a standard Brownian motion $\{B_t\}_{t\ge 0}$.

\item If $0\le\alpha_1\le1$ (equivalently $1/2\le p_1\le1$), as a function of $\alpha_1$,
$\E[V_sV_t]$
is symmetric about $\alpha_1=1/2$
(equivalently $p_1=3/4$).
\end{enumerate}

(ii) In the symmetric case $\alpha_1=\alpha_2\in[1/2,1)$,
the limiting difference process is, up to the deterministic factor $\sqrt2$, a one-dimensional noise-counterbalanced Brownian motion in the sense of Bertenghi and Rosales-Ortiz~\cite{bertenghi2022joint}, corresponding here
to the parameter $\alpha_1$.
\end{remark}

Since the limiting difference process is a centered continuous self-similar Gaussian process,
we apply the result of
Shibata--Shirai~\cite{shibata2025remark}.

\begin{proposition}
\label{prop:diff-lil}
The limiting difference process $\{V_t\}_{t\ge0}$ satisfies the law of the iterated logarithm:
\[
  \limsup_{t\to\infty}
  \pm\frac{V_t}{\sqrt{2t\log\log t}}
  =
  \begin{cases}
  \displaystyle
  \left(
    \frac{
      2\bigl\{1-\alpha_1-\alpha_2+(\alpha_1-\alpha_2)^2\bigr\}
    }{
      1-4\alpha_1\alpha_2
    }
  \right)^{1/2},
  & \mu<1/4,\\[4mm]
  \displaystyle
  \left(\frac{2}{1+2\alpha_1}\right)^{1/2},
  & \alpha_1=\alpha_2\in[1/2,1),
  \end{cases}
  \quad\text{a.s.}
\]
\end{proposition}

\begin{proof}
By \Cref{prop:diff-cov}, $\{V_t\}$ is a $1/2$-self-similar process:
$\E[V_{cs}V_{ct}]=c\,\E[V_s V_t]$ for $c>0,\ s,t>0$.
Hence condition~(i) of~\cite[Theorem~1.2]{shibata2025remark} holds with $\rho=1/2$.

For the covariance decay condition, define
$
h(x):=x^{-1/2}\E[V_1V_x]
$
for $x\ge 1$.
By~\Cref{prop:diff-cov},
\[
h(x)
=
\begin{cases}
O\!\left(x^{-(1/2-\sqrt{\mu})}\right),
& 0<\mu<1/4,\\[1mm]
O\!\left(x^{-1/2}\log x\right),
& \mu=0,\\[1mm]
O\!\left(x^{-1/2}\right),
& \mu<0,\\[1mm]
O\!\left(x^{-(\alpha_1+1/2)}\right),
& \alpha_1=\alpha_2\in[1/2,1),
\end{cases}
\qquad \text{as $x\to\infty$}.
\]
In particular, in all cases, $h(x)=O((\log x)^{-\eta})$ for every $\eta>0$, so condition~(ii)
of~\cite[Theorem~1.2]{shibata2025remark} holds.
Therefore
$\limsup_{t\to\infty}\pm V_t/\sqrt{2t\log\log t}=\sqrt{\E[V_1^2]}$ a.s.,
and setting $t=1$ in~\eqref{eq:diff-variance} completes the proof.
\end{proof}

\begin{proposition}
\label{prop:diff-local-lil}
For every fixed deterministic $t>0$, the limiting difference process $\{V_t\}_{t\ge0}$ satisfies the local law of the iterated logarithm:
\[
\limsup_{h\downarrow0}
\pm
\frac{V_{t+h}-V_t}
{\sqrt{4h\log\log(1/h)}}
=
1
\qquad\text{a.s.}
\]
\end{proposition}

\begin{proof}
Suppose first that $\mu<1/4$.
By \Cref{cor:diff-volterra},
\[
V_t
=
\dd^\top B_t
+
R_t,
\qquad
R_t
:=
\int_0^t
\frac1s \dd^\top AZ_s\,ds.
\]
Fix $t>0$. Since $Z$ has continuous sample paths, the function $s\mapsto \frac1s \dd^\top AZ_s$
is almost surely bounded on a neighbourhood of $t$. Therefore, there exists an a.s. finite random constant $C$ such that, for all sufficiently small $h>0$,
\[
|R_{t+h}-R_t|
\le
\int_t^{t+h}
\left|
\frac1s \dd^\top AZ_s
\right|\,ds
\le
Ch \qquad\text{a.s.},
\]
which implies that $R_{t+h}-R_t=O(h)$ as $h\downarrow0$.
Hence
\[
  \frac{R_{t+h}-R_t}
       {\sqrt{4h\log\log(1/h)}}
  =
  O\left(
    \frac{\sqrt h}{\sqrt{\log\log(1/h)}}
  \right)
  \longrightarrow0
  \qquad\text{a.s.}
\]
On the other hand, $\dd^\top B$ is a one-dimensional Brownian motion with variance parameter $2$. Therefore, by the classical local law of the iterated logarithm,
\[
\limsup_{h\downarrow0}
\pm
\frac{\dd^\top B_{t+h}-\dd^\top B_t}
{\sqrt{4h\log\log(1/h)}}
=
1
\qquad\text{a.s.}
\]

Suppose next that $\alpha_1=\alpha_2\in [1/2,1)$. By
\Cref{cor:diff-volterra},
\[
V_t
=
\sqrt2\,W_t
-
\alpha_1\int_0^t\frac{V_s}{s}\,ds.
\]
Since $V$ has continuous sample paths, the function $s\mapsto V_s/s$
is almost surely bounded on a neighbourhood of $t$. Therefore, $\int_t^{t+h}V_s/s\,ds=O(h)$ a.s. as $h\downarrow0$.
Moreover, since $\sqrt2\,W$ is a one-dimensional Brownian motion with variance parameter $2$, the same local law of the iterated logarithm above holds.

Finally, in both cases the finite-variation increment is $O(h)$ almost
surely, and
\[
  \frac{O(h)}
       {\sqrt{4h\log\log(1/h)}}
  =
  O\left(
    \frac{\sqrt h}{\sqrt{\log\log(1/h)}}
  \right)
  \longrightarrow0.
\]
Thus the finite-variation term is negligible at the LIL scale, which
proves the assertion.
\end{proof}

\subsection{Proof of \Cref{thm:diff-fclt}}
\label{subsec:proof-diff-fclt}

It remains to consider the symmetric case
$\alpha_1=\alpha_2\in[1/2,1)$, for which we use a scalar
martingale-transform argument.

\begin{lemma}
\label{lem:scalar-fundamental}
Let $\alpha_1=\alpha_2\in[1/2,1)$ and define
\[
\phi(m,\ell)
:=
\prod_{j=\ell}^{m-1}
\left(1-\alpha_1/j\right),
\qquad
1\le\ell\le m,
\]
with the convention $\phi(m,m)=1$.
Then the following assertions hold.

\begin{enumerate}
\item
For all $1\le\ell\le m$,
$
0<\phi(m,\ell)\le1.
$

\item
For every $0<\delta<T<\infty$,
\[
\sup_{\delta\le s\le t\le T}
\left|
\phi([nt],[ns])
-
\left(t/s\right)^{-\alpha_1}
\right|
\to0,
\qquad n\to\infty.
\]
\end{enumerate}
\end{lemma}
\begin{proof}
The first assertion is immediate from $\alpha_1\in[1/2,1)$. For the
second, the same argument as in \Cref{lem:fundamental} gives
\[
  \log\phi(m,\ell)
  =
  -\alpha_1\log\frac{m}{\ell}
  +O\left(\frac1\ell\right),
  \qquad m\ge\ell,
\]
uniformly in $m$. The assertion follows by taking
$m=[nt]$ and $\ell=[ns]$.
\end{proof}

\begin{proposition}
\label{lem:scalar-martingale-fclt}
Suppose that $\alpha_1=\alpha_2\in[1/2,1)$.
Set 
\[
\zeta_{n+1}:=\dd^\top D_{n+1},
\qquad
N_n:=\sum_{k=1}^{n-1}\zeta_{k+1},
\qquad
N_0:=0.
\]
Then
\[
\left\{
\frac{N_{[nt]}}{\sqrt n}
\right\}_{t\ge0}
\Longrightarrow
\{\sqrt2\,W_t\}_{t\ge0}
\]
in $D([0,\infty),\mathbb R)$, where
$\{W_t\}_{t\ge0}$ is a standard Brownian motion.
\end{proposition}
\begin{proof}
Set
$
  U_k^\pm:=S_k^{(1)}\pm S_k^{(2)}.
$
Since
$
  \E[U_{k+1}^\pm-U_k^\pm\mid\F_k]
  =
  \pm\frac{\alpha_1}{k}U_k^\pm
$
and $|U_{k+1}^\pm-U_k^\pm|\le2$, the standard second-moment
recursion gives
\[
  \E[(U_k^-)^2]=O(k),
  \qquad
  \E[(U_k^+)^2]
  =
  \begin{cases}
    O(k\log k), & \alpha_1=1/2,\\
    O(k^{2\alpha_1}), & \alpha_1>1/2.
  \end{cases}
\]
Consequently,
\begin{equation}
\label{eq:symmetric-second-moment}
  \E\bigl[(S_k^{(1)})^2+(S_k^{(2)})^2\bigr]
  =
  \begin{cases}
    O(k\log k), & \alpha_1=1/2,\\
    O(k^{2\alpha_1}), & \alpha_1>1/2.
  \end{cases}
\end{equation}

Since
\[
  \E[X_{k+1}^{(1)}-X_{k+1}^{(2)}\mid\F_k]
  =
  -\frac{\alpha_1}{k}
  \bigl(S_k^{(1)}-S_k^{(2)}\bigr),
\]
and
\[
  \E\bigl[
    (X_{k+1}^{(1)}-X_{k+1}^{(2)})^2
    \mid\F_k
  \bigr]
  =
  2
  -
  2\alpha_1^2
  \frac{S_k^{(1)}S_k^{(2)}}{k^2},
\]
we obtain
\begin{align*}
  \E[\zeta_{k+1}^2\mid\F_k]
  &=
  2
  -
  2\alpha_1^2
  \frac{S_k^{(1)}S_k^{(2)}}{k^2}
  -
  \alpha_1^2
  \frac{(S_k^{(1)}-S_k^{(2)})^2}{k^2}\\
  &=
  2
  -
  \frac{\alpha_1^2}{k^2}
  \bigl\{(S_k^{(1)})^2+(S_k^{(2)})^2\bigr\}.
\end{align*}
Moreover, $|\zeta_{k+1}|\le4$.

For
\[
  N^{(n)}(t):=\frac{N_{[nt]}}{\sqrt n},
  \qquad
  \nu_n(t):=([nt]-1)_+,
\]
the predictable quadratic variation is therefore
\[
  \langle N^{(n)}\rangle_t
  =
  \frac{2\nu_n(t)}{n}
  -
  \frac{\alpha_1^2}{n}
  \sum_{k=1}^{\nu_n(t)}
  \frac{(S_k^{(1)})^2+(S_k^{(2)})^2}{k^2}.
\]
By~\eqref{eq:symmetric-second-moment}, for every $T>0$,
\[
  \E\left[
    \frac1n
    \sum_{k=1}^{\nu_n(T)}
    \frac{(S_k^{(1)})^2+(S_k^{(2)})^2}{k^2}
  \right]
  =
  \begin{cases}
    O((\log n)^2/n), & \alpha_1=1/2,\\
    O(n^{2\alpha_1-2}), & \alpha_1>1/2,
  \end{cases}
\]
which tends to zero since $\alpha_1<1$. Hence, for every $T>0$,
\[
  \sup_{0\le t\le T}
  \bigl|\langle N^{(n)}\rangle_t-2t\bigr|
  \xrightarrow{\Prob}0.
\]
Finally, since $|\zeta_{k+1}|\le4$, the maximum jumps tend to zero and
the Lindeberg condition is immediate. The martingale functional central
limit theorem \cite[Theorem~2.1(ii)]{whitt2007proofs} therefore yields
the result.
\end{proof}

\begin{proof}[Proof of \Cref{thm:diff-fclt}]
Suppose first that $\mu<1/4$. Since the map
$
  x\longmapsto \dd^\top x
$
from $D([0,\infty),\R^2)$ to $D([0,\infty),\R)$ is continuous in the
Skorokhod $J_1$-topology, \Cref{thm:main} and the continuous mapping
theorem yield the assertion.

Suppose next that $\alpha_1=\alpha_2\in[1/2,1)$. Since
$\dd^\top A=-\alpha_1\dd^\top$, we have
\[
  \dd^\top\Phi(m,\ell)=\phi(m,\ell)\dd^\top.
\]
Applying $\dd^\top$ to the representation in~\Cref{lem:fundamental}
therefore gives
\begin{equation}
\label{eq:scalar-difference-representation}
  \mathcal D_m
  =
  \phi(m,1)\mathcal D_1
  +
  \sum_{k=1}^{m-1}\phi(m,k+1)\zeta_{k+1},
  \qquad
  \zeta_{k+1}:=\dd^\top D_{k+1}.
\end{equation}
By~\Cref{lem:scalar-martingale-fclt},
\[
  \left\{
    \frac{N_{[nt]}}{\sqrt n}
  \right\}_{t\ge0}
  \Longrightarrow
  \{\sqrt2\,W_t\}_{t\ge0},
  \qquad
  N_n:=\sum_{k=1}^{n-1}\zeta_{k+1}.
\]

We now apply the scalar counterpart of the martingale-transform
argument developed in \Cref{sec:proof}. By~\Cref{lem:scalar-fundamental},
$\phi$ satisfies the corresponding kernel approximation. Moreover,
the logarithmic estimate in its proof gives
\[
  \sup_{r\ge0}\sup_{2^r\le m\le2^{r+1}}
  \left\{
    \phi(m,2^r)+\phi(m,2^r)^{-1}
  \right\}
  <\infty.
\]
Since $0<\phi(m,\ell)\le1$, the dyadic decomposition used in
\Cref{lem:fundamental-maximal}, now in the scalar setting, gives
\begin{equation}
\label{eq:scalar-maximal}
  \E\max_{1\le m\le N}
  \left|
    \sum_{k=1}^{m-1}
    \phi(m,k+1)\zeta_{k+1}
  \right|^2
  \le CN.
\end{equation}

Fix $0<\varepsilon<T$ and set $a=\lfloor n\varepsilon\rfloor$.
Summation by parts as in~\eqref{sec43:abel}, together with
\Cref{lem:scalar-martingale-fclt} and the kernel approximation in
\Cref{lem:scalar-fundamental}, yields
\[
  \left\{
    \frac1{\sqrt n}
    \sum_{k=a}^{[nt]-1}
    \phi([nt],k+1)\zeta_{k+1}
  \right\}_{t\in[\varepsilon,T]}
  \Longrightarrow
  \left\{
    \sqrt2\int_\varepsilon^t
    \left(\frac ts\right)^{-\alpha_1}\,dW_s
  \right\}_{t\in[\varepsilon,T]}.
\]

It remains to control the contribution from the neighborhood of the
origin. By~\eqref{eq:scalar-maximal},
\[
  \E\max_{1\le m\le a}
  \left|
    \sum_{k=1}^{m-1}
    \phi(m,k+1)\zeta_{k+1}
  \right|^2
  \le Ca.
\]
For $m\ge a$, the multiplicative property of $\phi$ gives
\[
  \sum_{k=1}^{a-1}\phi(m,k+1)\zeta_{k+1}
  =
  \phi(m,a)
  \sum_{k=1}^{a-1}\phi(a,k+1)\zeta_{k+1}.
\]
Since $0<\phi(m,a)\le1$, it follows that
\[
  \E\sup_{0\le t\le T}
  \left|
    Y_n(t)-Y_n^{(\varepsilon)}(t)
  \right|^2
  \le
  \frac{Ca}{n}
  \le C\varepsilon,
\]
where $Y_n^{(\varepsilon)}$ denotes the process with the terms before
time $\varepsilon$ removed. Thus the contribution from the neighborhood
of the origin is negligible as $\varepsilon\downarrow0$.

Now set
\[
  V_t
  :=
  \sqrt2\int_0^t
  \left(\frac ts\right)^{-\alpha_1}\,dW_s,
  \qquad t>0,
  \qquad
  V_0:=0.
\]
For $t\ge\varepsilon$,
\[
  V_t
  -
  \sqrt2\int_\varepsilon^t
  \left(\frac ts\right)^{-\alpha_1}\,dW_s
  =
  \left(\frac{t}{\varepsilon}\right)^{-\alpha_1}V_\varepsilon.
\]
Since $(t/\varepsilon)^{-\alpha_1}\le1$ and $V$ has a continuous
modification, the truncated limiting processes converge uniformly on
compact intervals to $V$ as $\varepsilon\downarrow0$. Hence the
converging-together argument 
yields
\[
  \left\{
    \frac1{\sqrt n}
    \sum_{k=1}^{[nt]-1}
    \phi([nt],k+1)\zeta_{k+1}
  \right\}_{t\ge0}
  \Longrightarrow
  \{V_t\}_{t\ge0}.
\]

Finally, since $0<\phi(m,1)\le1$,
\[
  \sup_{1/n\le t\le T}
  \frac{|\phi([nt],1)\mathcal D_1|}{\sqrt n}
  \le
  \frac{|\mathcal D_1|}{\sqrt n}
  \longrightarrow0.
\]
Combining this with~\eqref{eq:scalar-difference-representation} gives
the assertion.
\end{proof}

\section{First separation time of the two walks}
\label{sec:exit}

We finally study the first separation time of the two interacting walks.
This can be expressed as the first exit time of the difference process. 
Throughout this section, we assume either $\mu<1/4$ or $\alpha_1=\alpha_2\in[1/2,1)$, as in \Cref{thm:diff-fclt}.

For each positive integer $c>0$, define
\begin{equation}
\label{eq:exit-discrete}
  \sigma_c:=\inf\bigl\{n\ge1:\bigl|\mathcal{D}_n\bigr|\ge 2c\bigr\}.
\end{equation}
If the threshold $2c$ is never reached, we set $\sigma_c=\infty$.
Rescaling time by $c^2$ yields the identity
\begin{equation}
\label{eq:exit-rescaling}
  \frac{\sigma_c}{c^2}
  =\inf\Biggl\{t>0:
  \Biggl|\frac{\mathcal{D}_{[c^2 t]}}{c}\Biggr|\ge2\Biggr\},
\end{equation}
with the usual convention $\inf\emptyset=\infty$.

\begin{proposition}
\label{prop:exit-convergence}
As $c\to\infty$,
\begin{equation}
\label{eq:exit-limit}
  \frac{\sigma_c}{c^2}
  \ \Longrightarrow\
  \tau_{\alpha_1,\alpha_2}
  :=\inf\bigl\{t>0:\bigl|V_t\bigr|\ge 2\bigr\}.
\end{equation}
\end{proposition}
\begin{proof}
Write $X^{(c)}(t):=\mathcal D_{[c^2t]}/c$.
By \Cref{thm:diff-fclt}, we have $X^{(c)}\Rightarrow V$ in $D([0,\infty),\R)$.
Define the first-exit map $\tau:D([0,\infty),\R)\to [0,\infty]$ by $\tau(x):=\inf\{t>0:|x(t)|\ge2\}$, with $\inf\emptyset=\infty$.
By \eqref{eq:exit-rescaling}, $\sigma_c/c^2=\tau(X^{(c)})$.

We first identify a continuity set for $\tau$. Let
\[
\mathcal C
:=
\left\{
x\in C([0,\infty),\R):
\begin{array}{l}
|x(0)|<2,\ \tau(x)<\infty,\\[1mm]
\text{and for every }\varepsilon>0
\text{ there exists }
t\in(\tau(x),\tau(x)+\varepsilon)\\
\text{such that }|x(t)|>2
\end{array}
\right\}.
\]
We show that $\tau$ is continuous at every $x\in\mathcal C$
with respect to uniform convergence on compact intervals. 
Let $x_n\to x$ uniformly on compact intervals and write
$\tau:=\tau(x)<\infty$.
Fix $\varepsilon>0$ with $\varepsilon<\tau$.
Since $|x(t)|<2$ for $0\le t\le \tau-\varepsilon$,
and $x$ is continuous, we have $\max_{0\le t\le\tau-\varepsilon}|x(t)|<2$. Therefore, there exists $\delta>0$ such that $|x(t)|\le2-\delta$ for $0\le t\le\tau-\varepsilon$.
By the assumption of uniform convergence on compact intervals, $\sup_{0\le t\le\tau-\varepsilon}|x_n(t)-x(t)|<\delta$
for all sufficiently large $n$. Hence
\[
  |x_n(t)|
  \le
  |x(t)|+|x_n(t)-x(t)|
  <
  2,
  \qquad
  0\le t\le\tau-\varepsilon,
\]
which implies that $\tau(x_n)\ge\tau-\varepsilon$.
On the other hand, since $x\in\mathcal C$, there exists
$t_\varepsilon\in(\tau,\tau+\varepsilon)$ such that $|x(t_\varepsilon)|>2$.
Hence there exists $\delta>0$ such that $|x(t_\varepsilon)|\ge2+\delta$.
By the assumption of uniform convergence on compact intervals, $\sup_{0\le t\le\tau+\varepsilon}
|x_n(t)-x(t)|<\delta$ for all sufficiently large $n$. Therefore,
\[
|x_n(t_\varepsilon)|
\ge
|x(t_\varepsilon)|
-
|x_n(t_\varepsilon)-x(t_\varepsilon)|
>
2,
\]
and hence $\tau(x_n)\le t_\varepsilon<\tau+\varepsilon$.
Consequently, 
$\tau(x_n)\to\tau(x)$ as $n\to\infty$.

We next claim that $\Prob(V\in\mathcal C)=1$.
By~\Cref{prop:diff-lil}, $\tau:=\tau(V)<\infty$ almost surely, while
$\tau>0$ since $V_0=0$ and $V$ has continuous sample paths.
Moreover, $V$ is adapted to the augmented filtration of the corresponding
driving Brownian motion, and hence $\tau$ is a stopping time with respect to this filtration.

In either regime of~\Cref{thm:diff-fclt}, the strong Markov property
implies that the increment of the driving Brownian motion after $\tau$
is a fresh Brownian motion. Thus, by the representations in
\Cref{cor:diff-volterra}, there exists a standard Brownian motion
$\widetilde W$ such that, almost surely,
\[
  \left|
    V_{\tau+h}-V_\tau-\sqrt2\,\widetilde W_h
  \right|
  \le C(\omega)h
\]
for all sufficiently small $h>0$, where $C(\omega)<\infty$. Indeed, since $\tau>0$, the
corresponding drift integrand is continuous, and hence bounded, on
$[\tau,\tau+1]$ for almost every sample path.
Hence
\[
  \frac{C(\omega)h}
       {\sqrt{4h\log\log(1/h)}}
  =
  \frac{C(\omega)\sqrt h}
       {2\sqrt{\log\log(1/h)}}
  \longrightarrow0
  \qquad\text{a.s.}
\]
Therefore, the local law of the iterated logarithm yields
\[
  \limsup_{h\downarrow0}
  \frac{V_{\tau+h}-V_\tau}
       {\sqrt{4h\log\log(1/h)}}
  =1,
  \qquad
  \liminf_{h\downarrow0}
  \frac{V_{\tau+h}-V_\tau}
       {\sqrt{4h\log\log(1/h)}}
  =-1
  \qquad\text{a.s.}
\]

Thus $V_{\tau+h}-V_\tau$ takes both signs arbitrarily close to
$h=0$. By continuity of $V$, we have $|V_\tau|=2$ on $\{\tau<\infty\}$. It follows that, for every
$\varepsilon>0$, there exists $t\in(\tau,\tau+\varepsilon)$ such that
$|V_t|>2$. Hence $V\in\mathcal C$ almost surely.

Since $V$ has continuous sample paths almost surely, $J_1$-convergence
to $V$ is equivalent to uniform convergence on compact intervals.
Since $\Prob(V\in\mathcal C)=1$ and $\tau$ is continuous on
$\mathcal C$, the continuous mapping theorem yields the assertion.
\end{proof}

\begin{proposition}
\label{prop:uniform-moment}
For every $p>1$, there exists a constant $C_p<\infty$, depending only on
$p$ and $(\alpha_1,\alpha_2)$, such that
\begin{equation}
\label{eq:uniform-moment}
 \sup_{c\ge1}
 \left\|
 \frac{\sigma_c}{c^2}
 \right\|_{L^p}
 \le C_p.
\end{equation}
In particular, the family
$\left\{\sigma_c/c^2\right\}_{c\ge1}$ is uniformly integrable.
\end{proposition}
\begin{proof}
We retain the notation $S_n$, $D_{n+1}$ and $\Phi$ from Section 4, and recall
that $\mathcal{D}_n=\dd^\top S_n$ with $\dd=(1,-1)^\top$.
Write $\E_n[\cdot]=\E[\cdot\mid\F_n]$. As in the proof of Proposition 4.2,
\[
 Q_k:=\E_k[D_{k+1}D_{k+1}^{\top}]
 =\diag\left(
  1-\alpha_1^2\left(\frac{S_k^{(2)}}{k}\right)^2,
  1-\alpha_2^2\left(\frac{S_k^{(1)}}{k}\right)^2
 \right).
\]

Choose an integer $L\ge256$ divisible by $8$, to be enlarged below. For
integers $c,j\ge1$, set $m=Lc^2$, $n=jm$ and $N=n+m$. Iterating the recursion in
Lemma 4.3 gives
\[
 \mathcal{D}_N=z+Y,\qquad
 z=\dd^\top\Phi(N,n)S_n,\qquad
 Y=\sum_{k=n}^{N-1}w_k^\top D_{k+1},\qquad
 w_k=\Phi(N,k+1)^\top \dd.
\]
For the Euclidean operator norm,
$\|A\|=\max\{|\alpha_1|,|\alpha_2|\}\le1$. Thus, for $n\le k<N$,
\[
 \|\Phi(N,k+1)\|
 \le\prod_{\ell=k+1}^{N-1}\left(1+\frac{1}{\ell}\right)
 =\frac{N}{k+1}\le\frac{n+m}{n+1}<2,
\]
since $m\le n$.
Consequently, $Y$ is a sum of $m$ uniformly bounded martingale differences,
with $\E_n[Y]=0$ and $z=\E_n[\mathcal{D}_N]$.
Conditional orthogonality and Burkholder's inequality
\cite[Theorem~2.10]{hall1980martingale} yield
\[
 \E_n[Y^2]=\E_n\left[\sum_{k=n}^{N-1}w_k^\top Q_k w_k\right],
 \qquad \E_n[Y^4]\le C m^2.
\]
The constant $C$ is independent of $c,j$ and $L$.

We claim that, for some $\kappa>0$ depending only on
$(\alpha_1,\alpha_2)$,
\[
 H:=\E_n[\mathcal{D}_N^2]=z^2+\E_n[Y^2]
 \ge\kappa(z^2+m)\qquad\text{on }\{\sigma_c>n\}.
\]
Put $\delta=2-\alpha_1^2-\alpha_2^2$ and distinguish two cases.

\smallskip
\noindent\emph{Case 1: $\delta>0$.}
This includes the symmetric case
$\alpha_1=\alpha_2\in[1/2,1)$, for which
$\delta=2(1-\alpha_1^2)>0$.
For $N-m/8\le k<N$, expanding the same product and using $N\ge2m$ give
\[
 \|\Phi(N,k+1)-I\|
 \le\prod_{\ell=k+1}^{N-1}\left(1+\frac{1}{\ell}\right)-1
 =\frac{N-k-1}{k+1}\le\frac{1}{15}.
\]
Hence $\|w_k-\dd\|\le\sqrt{2}/15<1/2$, so $w_{k,i}^2\ge1/4$ for
$i=1,2$. As $Q_k$ is diagonal and $\tr Q_k\ge\delta$, we obtain
\[
 \E_n[Y^2]\ge\frac{\delta m}{32}.
\]
Since $\delta\le2$,
\[
 H
 =z^2+\E_n[Y^2]
 \ge z^2+\frac{\delta m}{32}
 \ge\frac{\delta}{32}(z^2+m).
\]
Thus the claim holds with $\kappa=\delta/32$, even without restricting
to $\{\sigma_c>n\}$.

\noindent\emph{Case 2: $\delta=0$.}
Under the standing assumptions, $(\alpha_1,\alpha_2)$ is either
$(1,-1)$ or $(-1,1)$. In this case $A^\top=-A$ and $A^2=-I$.
Consequently, each factor $I+A/k$ is a scalar at least $1$ times an
orthogonal matrix, and therefore $\|w_k\|^2\ge2$.
Set $R_n=S_n^{(1)}+S_n^{(2)}$ and work on $\{\sigma_c>n\}$, where
$|\mathcal{D}_n|<2c$.

If $|R_n|\le n$, then $|S_n^{(i)}|\le n/2+c$. Bounded increments give,
for $n\le k<n+m/8$ and $i=1,2$,
\[
 \frac{|S_k^{(i)}|}{k}
 \le\frac{n/2+c+(k-n)}{n}
 \le\frac{5}{8}+\frac{1}{L}\le\frac{3}{4}.
\]
Hence $Q_k\succeq I/4$ throughout this interval, and
\[
 \E_n[Y^2]\ge\frac{m}{8}\cdot\frac{1}{4}\cdot2=\frac{m}{16}.
\]
Therefore,
\[
 H
 =z^2+\E_n[Y^2]
 \ge z^2+\frac{m}{16}
 \ge\frac1{16}(z^2+m).
\]
If $|R_n|>n$, use the factorization
\[
 \Phi(N,n)=\rho(\cos\vartheta I+\sin\vartheta A),\qquad
 \rho=\prod_{k=n}^{N-1}\sqrt{1+\frac{1}{k^2}}\ge1,\qquad
 \vartheta=\sum_{k=n}^{N-1}\arctan\frac{1}{k}.
\]
Since $m\le n$, using $\arctan u\ge u/2$ for $u\in[0,1]$ and
$\log(1+x)\ge x/2$ for $x\in[0,1]$, we obtain
\[
 \vartheta
 \ge\frac12\sum_{k=n}^{N-1}\frac1k
 \ge\frac12\log\left(1+\frac{m}{n}\right)
 \ge\frac{m}{4n}.
\]
Moreover,
\[
 \vartheta
 \le\sum_{k=n}^{N-1}\frac1k
 \le\frac{m}{n}\le1.
\]
Hence, since $\sin\vartheta\ge\vartheta/2$ on $[0,1]$,
\[
 \sin\vartheta\ge\frac{m}{8n}.
\]
Since $\dd^\top A=\alpha_1(1,1)$,
$z=\rho(\cos\vartheta \mathcal{D}_n+\alpha_1\sin\vartheta R_n)$.
Using $|\alpha_1|=1$, $L\ge256$ and $c\ge1$, we obtain
\[
 |z|\ge\sin\vartheta |R_n|-|\mathcal{D}_n|
 \ge\frac{m}{8}-2c\ge\frac{m}{16}\ge\sqrt m.
\]
Thus $z^2\ge m$, and hence
\[
 H\ge z^2\ge\frac12(z^2+m)\ge\frac1{16}(z^2+m).
\]
Therefore, in both cases the claim holds with $\kappa=1/16$.

\smallskip
In both cases, the fourth-moment bound gives, on $\{\sigma_c>n\}$,
\[
 \E_n[\mathcal{D}_N^4]
 \le8z^4+8\E_n[Y^4]
 \le C(z^4+m^2)\le K H^2,
\]
where $K<\infty$ depends only on the parameters. On the event $\{\mathcal{D}_N^2<H/2\}$, we have
$\mathcal{D}_N^2\le H/2$. Hence
\[
 \frac{H}{2}
 \le
 \E_n\left[
   \mathcal{D}_N^2
   \mathbf1_{\{\mathcal{D}_N^2\ge H/2\}}
 \right]
 \le
 \E_n[\mathcal{D}_N^4]^{1/2}
 \mathbb{P}\left(
   \mathcal{D}_N^2\ge\frac{H}{2}
   \,\middle|\,\F_n
 \right)^{1/2}.
\]
Therefore,
\[
 \mathbb{P}\left(
   \mathcal{D}_N^2\ge\frac{H}{2}
   \,\middle|\,\F_n
 \right)
 \ge
 \frac{H^2}{4\E_n[\mathcal{D}_N^4]}
 \ge\frac{1}{4K}
 =:p_0>0.
\]
Enlarge $L$, still keeping it divisible by $8$, so that $\kappa L\ge8$.
Then $H/2\ge\kappa m/2\ge4c^2$ on $\{\sigma_c>n\}$, and hence
\[
\mathbb{P}\bigl(|\mathcal{D}_{(j+1)m}|\ge2c\mid\F_{jm}\bigr)\ge p_0
 \qquad\text{on }\{\sigma_c>jm\},\qquad j\ge1.
\]
Using
$\{\sigma_c>(j+1)m\}\subset
\{\sigma_c>jm\}\cap\{|\mathcal{D}_{(j+1)m}|<2c\}$ and iterating, we obtain
\begin{equation}\label{eq:tail}
 \mathbb{P}(\sigma_c>jm)\le(1-p_0)^{j-1},\qquad j\ge1.
\end{equation}
Finally, for every $p>1$, integration of this tail bound yields
\[
 \E\left[\left(\frac{\sigma_c}{m}\right)^p\right]
 =p\int_0^\infty t^{p-1}\mathbb{P}(\sigma_c>tm) dt
 \le1+\sum_{j=1}^{\infty}\bigl((j+1)^p-j^p\bigr)(1-p_0)^{j-1}<\infty,
\]
uniformly in $c$. Since $m=Lc^2$, this proves \eqref{eq:uniform-moment};
uniform integrability follows from the uniform $L^p$ bound.
\end{proof}

\begin{lemma}
\label{lem:exit-selfsimilar-sup}
Let $\{V_t\}_{t\ge0}$ be a continuous $H$-self-similar process with
$H>0$, that is,
\[
\{V_{ct}\}_{t\ge0}
\stackrel{\mathrm{law}}{=}
\{c^H V_t\}_{t\ge0}
\qquad\text{for every }c>0.
\]
Define $M:=\sup_{0\le u\le1}|V_u|$ and $\tau^{(\ell)}:=\inf\{t>0:|V_t|\ge \ell\}$ for every $\ell>0$. Suppose that $\Prob (M>0)=1$.
Then
\begin{equation*}
\tau^{(\ell)}
\stackrel{\mathrm{law}}{=}
\left(\frac{\ell}{M}\right)^{1/H}.
\end{equation*}
\end{lemma}

\begin{proof}
For every $t>0$, by continuity of $V$, $\{\tau^{(\ell)}\le t\}
=\left\{\sup_{0\le s\le t}|V_s|\ge \ell\right\}$.
By the change of variables $s=tu$ and the $H$-self-similarity of $V$,
\[
\sup_{0\le s\le t}|V_s|
=
\sup_{0\le u\le1}|V_{tu}|
\stackrel{\mathrm{law}}{=}
t^H\sup_{0\le u\le1}|V_u|
=
t^H M.
\]
Therefore,
\begin{align*}
\Prob(\tau^{(\ell)}\le t)=
\Prob\left(
\sup_{0\le s\le t}|V_s|\ge \ell
\right)=
\Prob\left(
\left(\frac{\ell}{M}\right)^{1/H}\le t
\right).
\end{align*}
Thus $\tau^{(\ell)}$ and $(\ell/M)^{1/H}$ have the same distribution.
\end{proof}

\begin{theorem}
\label{prop:exit-expectation}
The following statements hold.
\begin{enumerate}
\item
For every $q>0$, $\E[\tau_{\alpha_1,\alpha_2}^q]<\infty.$
In particular, $m_{\alpha_1,\alpha_2}:=\E[\tau_{\alpha_1,\alpha_2}]<\infty$.
\item
For the limiting process $V$, put $M:=\sup_{0\le u\le1}|V_u|$ and $\tau^{(\ell)}:=\inf\{t>0:|V_t|\ge \ell\}$ for every $\ell>0$.
Then 
\[
\E[\tau^{(\ell)}]=\ell^2\E[M^{-2}].
\]
In particular, $m_{\alpha_1,\alpha_2}=4\E[M^{-2}]$.
\item
For every $q>0$, 
as $c\to\infty$,
\begin{equation}
\label{eq:sigma-c-moment-asymptotic}
\E[\sigma_c^q]
=
c^{2q}\E[\tau_{\alpha_1,\alpha_2}^q]
+
o(c^{2q}).
\end{equation}
In particular,
\begin{equation}
\label{eq:sigma-c-asymptotic}
    \E[\sigma_c]
    =
    m_{\alpha_1,\alpha_2}\,c^2
    +
    o(c^2).
\end{equation}
\end{enumerate}
\end{theorem}
\begin{proof}
Fix $q>0$ and choose $p>\max\{1,q\}$.
We first prove~(3).
By~\Cref{prop:uniform-moment},
\[
\sup_{c\ge1}
\E\left[
\left\{
\left(\frac{\sigma_c}{c^2}\right)^q
\right\}^{p/q}
\right]
=
\sup_{c\ge1}
\E\left[
\left(\frac{\sigma_c}{c^2}\right)^p
\right]
<\infty.
\]
Since $p/q>1$, the family
$
\left\{
\left(\sigma_c/c^2\right)^q
\right\}_{c\ge1}
$
is uniformly integrable.
Hence, by~\Cref{prop:exit-convergence}, the continuous mapping theorem,
and~\cite[Theorem~3.5]{billingsley1999convergence},
$
\lim_{c\to\infty}
\E\left[
\left(\frac{\sigma_c}{c^2}\right)^q
\right]
=
\E[\tau_{\alpha_1,\alpha_2}^q],
$
which proves~(3). In particular,
$
\E[\tau_{\alpha_1,\alpha_2}^q]<\infty,
$
which proves~(1). Taking $q=1$ gives~\eqref{eq:sigma-c-asymptotic}.

Finally, we prove~(2). By~\Cref{prop:diff-cov},
$\{V_t\}_{t\ge0}$ is a $1/2$-self-similar process. Moreover,
by~\Cref{prop:diff-cov}, $\E[V_1^2]>0$. Hence, $V_1$ is a
non-degenerate centered Gaussian, and consequently,
$\mathbb{P}(M>0)=1$.
\Cref{lem:exit-selfsimilar-sup} therefore yields
$
\tau^{(\ell)}\overset{\mathrm{law}}{=}\ell^2M^{-2}.
$
Taking $\ell=2$ and using part~(1), we obtain
$
\E[M^{-2}]
<\infty.
$
Hence,
$
\E[\tau^{(\ell)}]
=
\ell^2\E[M^{-2}].
$
This completes the proof.
\end{proof}

\begin{remark}
\label{rem:exit-expectation}
(i) By the covariance form in~\Cref{prop:diff-cov}, we have $m_{\alpha_1,\alpha_2}=m_{\alpha_2,\alpha_1}$.

(ii) If $\alpha_2=0$ and $\alpha_1\in\{0,1\}$, then $V\stackrel{\mathrm{law}}{=}\sqrt{2}\,B$ for a
standard Brownian motion $B$. 
Since the expected first exit time of $B$ from $(-a,a)$ is $a^2$, we have $m_{\alpha_1,0}=2$ (see also~\Cref{fig:expected_exit_time_curves_selfsimilar_fixmu}) and
$
  \E[\sigma_c]=2c^2+o(c^2).
$
\end{remark}

We conclude this section by establishing monotonicity results for the
limiting exit times.
The proof is based on Anderson's inequality, which allows us to compare
the corresponding Gaussian measures.
\begin{lemma}[Anderson's inequality~\cite{anderson1955integral}]
\label{lem:anderson}
Let $P$ be a centered Gaussian measure on $\R^d$, and let $A\subset\R^d$ be a convex, centrally symmetric Borel set. Then for any $x\in\R^d$, $P(A+x)\le P(A)$.
\end{lemma}

\begin{corollary}
\label{cor:anderson-cov}
Let $P_\Sigma$ denote the centered Gaussian measure on $\R^d$ with covariance matrix $\Sigma$.
If $\Sigma_1\preceq\Sigma_2$ in the Loewner order (i.e.\ $\Sigma_2-\Sigma_1$ is positive semidefinite)
and $A\subset\R^d$ is a convex, centrally symmetric Borel set, then $P_{\Sigma_2}(A)\le P_{\Sigma_1}(A)$.
\end{corollary}

\begin{proof}
Write $\Sigma_2=\Sigma_1+\Gamma$ with $\Gamma\succeq0$. If $G\sim P_{\Sigma_1}$ and $H\sim P_\Gamma$ are
independent, then $G+H\sim P_{\Sigma_2}$. Conditioning on $H=x$ and applying
\Cref{lem:anderson} yields
$\Prob(G+H\in A\mid H=x)=P_{\Sigma_1}(A-x)\le P_{\Sigma_1}(A)$.
Taking expectations gives the claim.
\end{proof}

For a fixed $\mu<1/4$, define
\[
\Lambda_\mu
:=
\left\{
\frac{\nu(1-\nu)}{1-4\mu}
:
\alpha_1,\alpha_2\in[-1,1],\;\mu=\alpha_1\alpha_2 ,\; \nu=\alpha_1+\alpha_2
\right\}.
\]
For each \(\lambda\in\Lambda_\mu\), let
\(\{V^{(\mu,\lambda)}_t\}_{t\ge0}\)
denote a centered continuous Gaussian process with $V^{(\mu,\lambda)}_0=0$ and covariance kernel given in~\Cref{prop:diff-cov}. Furthermore, define $\tau_{\mu,\lambda}:=\inf\left\{t>0:|V_t^{(\mu,\lambda)}|\ge2\right\}$.
For nonnegative random variables $X$ and $Y$, 
$X\le_{\mathrm{st}}Y$ means $\Prob(X>T)\le\Prob(Y>T)$ 
for every $T\ge0$.
\begin{theorem}
\label{prop:expect-tau}
The following stochastic comparisons hold.
\begin{enumerate}
\item
Fix $-1\le\mu<1/4$.
If $\lambda,\lambda'\in\Lambda_\mu$ satisfy $\lambda\le\lambda'$, then
\[
\tau_{\mu,\lambda}
\le_{\mathrm{st}}
\tau_{\mu,\lambda'}.
\]
\item
If $-1/2<\alpha\le\alpha'<1$, then
\[
\tau_{\alpha,\alpha}
\le_{\mathrm{st}}
\tau_{\alpha',\alpha'}.
\]
\end{enumerate}
In particular, the same orders hold for expectations and all positive moments.
\end{theorem}

\begin{proof}
We first prove~(1). We begin by showing that $C_\mu$ in~\Cref{prop:diff-cov}
is a covariance kernel. Let $B$ be a standard Brownian motion and set
\[
  U_t^{(\mu)}
  :=
  \sqrt{1-4\mu}
  \int_0^t
  s_\mu\!\left(\log\frac{t}{r}\right)\,dB_r.
\]
The stochastic integral is well-defined since, by
\eqref{function:convenient_e-matrix},
\[
  |s_\mu(x)|
  =
  O\!\left((1+x)e^{\sqrt{\mu_+}x}\right),
  \qquad
  \mu_+:=\max\{\mu,0\},
\]
and $2\sqrt{\mu_+}<1$.

For $0<s\le t$, It\^o's isometry and the change of variables
$r=se^{-x}$ give, with $h=\log(t/s)$,
\[
  \E[U_s^{(\mu)}U_t^{(\mu)}]
  =
  (1-4\mu)s
  \int_0^\infty
  e^{-x}s_\mu(x)s_\mu(x+h)\,dx.
\]
Using
$
  s_\mu(x+h)
  =
  s_\mu(x)c_\mu(h)+c_\mu(x)s_\mu(h),
$
we compute the two integrals appearing above. Suppose first that
$\mu>0$ and put $\beta=\sqrt{\mu}\in(0,1/2)$. Then
\[
  s_\mu(x)=\frac{\sinh(\beta x)}{\beta},
  \qquad
  c_\mu(x)=\cosh(\beta x).
\]
Since
\[
  \int_0^\infty e^{-x}\sinh(2\beta x)\,dx
  =
  \frac{2\beta}{1-4\beta^2},
\]
we obtain
\[
  \int_0^\infty e^{-x}c_\mu(x)s_\mu(x)\,dx
  =
  \frac{1}{2\beta}
  \int_0^\infty e^{-x}\sinh(2\beta x)\,dx
  =
  \frac{1}{1-4\mu}.
\]
Moreover,
\[
  \int_0^\infty e^{-x}\sinh^2(\beta x)\,dx
  =
  \frac12
  \left(
    \frac{1}{1-4\beta^2}-1
  \right)
  =
  \frac{2\mu}{1-4\mu},
\]
and hence
\[
  \int_0^\infty e^{-x}s_\mu(x)^2\,dx
  =
  \frac{2}{1-4\mu}.
\]
The cases $\mu=0$ and $\mu<0$ follow similarly, using
$s_0(x)=x$ and the corresponding trigonometric expressions,
respectively. Together with
$
  c_\mu(x)^2-\mu s_\mu(x)^2=1,
$
these identities yield
\[
  \E[U_s^{(\mu)}U_t^{(\mu)}]
  =
  s\bigl(2c_\mu(h)+s_\mu(h)\bigr)
  =
  C_\mu(s,t),
\]
so $C_\mu$ is a covariance kernel.

Now let $\lambda,\lambda'\in \Lambda_\mu$ with $\lambda<\lambda'$.
Define $K_{\mu,\lambda}(s,t):=\E\bigl[V^{(\mu,\lambda)}_sV^{(\mu,\lambda)}_t\bigr]$. 
For any $0<t_1<\cdots<t_n$, set
$
K_{\mu,\lambda}^{(n)}
:=
\bigl(K_{\mu,\lambda}(t_i,t_j)\bigr)_{1\le i,j\le n}.
$
By \Cref{prop:diff-cov},
\[
K_{\mu,\lambda}^{(n)}
-
K_{\mu,\lambda'}^{(n)}
=
(\lambda'-\lambda)
\bigl(C_\mu(t_i,t_j)\bigr)_{1\le i,j\le n}
\succeq0.
\]
Hence $K_{\mu,\lambda'}^{(n)}\preceq K_{\mu,\lambda}^{(n)}$
in the Loewner order.
Moreover,  $(V_{t_1}^{(\mu,\lambda)},\ldots,V_{t_n}^{(\mu,\lambda)})
\sim N(0,K_{\mu,\lambda}^{(n)})$ and
the set $[-\ell,\ell]^n$, $\ell>0$, is convex and symmetric. Hence
\Cref{cor:anderson-cov} yields
\begin{align}
\Prob\bigl(
|V_{t_1}^{(\mu,\lambda')}|\le \ell,\ldots,
|V_{t_n}^{(\mu,\lambda')}|\le \ell
\bigr)
\ge
\Prob\bigl(
|V_{t_1}^{(\mu,\lambda)}|\le \ell,\ldots,
|V_{t_n}^{(\mu,\lambda)}|\le \ell
\bigr).
\label{eq:anderson-fdd-general}
\end{align}
Fix $T>0$ and let
$
\Pi_m
:=
\left\{
kT/2^m:k=0,\ldots,2^m
\right\}.
$
Applying \eqref{eq:anderson-fdd-general} to the positive time points of
$\Pi_m$, and then letting $m\to\infty$, path-continuity yields
\[
\Prob\left(
\sup_{0\le s\le T}
|V_s^{(\mu,\lambda')}|
\le \ell
\right)
\ge
\Prob\left(
\sup_{0\le s\le T}
|V_s^{(\mu,\lambda)}|
\le \ell
\right).
\]
Now let $\ell_m:=2-1/m,\ m\ge1$.
Since
\[
\left\{
\sup_{0\le s\le T}
|V_s^{(\mu,\lambda)}|
<2
\right\}
=
\bigcup_{m=1}^\infty
\left\{
\sup_{0\le s\le T}
|V_s^{(\mu,\lambda)}|
\le \ell_m
\right\}
=
\{\tau_{\mu,\lambda}>T\},
\]
and similarly for $V^{(\mu,\lambda')}$, continuity of probability from below
implies the first assertion.

We next prove~(2).
Suppose that $-1/2<\alpha<\alpha'<1$.
By~\eqref{eq:diff-cov-symmetric}, for every
$\alpha\in(-1/2,1)$ and $0<s\le t$,
\[
  \E\!\left[V_s^{(\alpha)}V_t^{(\alpha)}\right]
  =
  \frac{2s}{2\alpha+1}
  \left(\frac{t}{s}\right)^{-\alpha}.
\]

Define the Lamperti transform
$
  X_u^{(\alpha)}
  :=
  e^{-u/2}V_{e^u}^{(\alpha)}
$
for $u\in \R$.
Then
\[
  \E\!\left[X_u^{(\alpha)}X_v^{(\alpha)}\right]
  =
  \frac{1}{\alpha+\frac12}
  e^{-(\alpha+\frac12)|u-v|},
\]
so $X^{(\alpha)}$ is a stationary Gaussian process with 
covariance function
\[
  r_\alpha(h)
  :=
  \E\!\left[X_0^{(\alpha)}X_h^{(\alpha)}\right]
  =
  \frac{1}{2\pi}
  \int_{\R} e^{i\xi h}f_\alpha(\xi)\,d\xi,
  \qquad
  f_\alpha(\xi)
  =
  \frac{2}
  {\left(\alpha+\frac 12\right)^2+\xi^2}.
\]
Since $\alpha<\alpha'$, we have
$
  f_{\alpha'}(\xi)\le f_\alpha(\xi)
$
for $\xi\in\R$.
Hence, by Bochner's theorem, the covariance kernel
$
  \E\!\left[X_u^{(\alpha)}X_v^{(\alpha)}\right]
  -
  \E\!\left[X_u^{(\alpha')}X_v^{(\alpha')}\right]
$
is positive semidefinite.

For arbitrary $0<t_1<\cdots<t_n$, let
$
  K_\alpha^{(n)}
  :=
  \operatorname{Cov}
  \bigl(
    V_{t_1}^{(\alpha)},\ldots,V_{t_n}^{(\alpha)}
  \bigr).
$
Since
$
  V_t^{(\alpha)}
  =
  \sqrt{t}\,X_{\log t}^{(\alpha)},
$
it follows that
$
  K_{\alpha'}^{(n)}
  \preceq
  K_\alpha^{(n)}
$
in the Loewner order.
Therefore, by~\Cref{cor:anderson-cov},
\[
  \Prob\bigl(
    |V_{t_1}^{(\alpha)}|\le\ell,\ldots,
    |V_{t_n}^{(\alpha)}|\le\ell
  \bigr)
  \le
  \Prob\bigl(
    |V_{t_1}^{(\alpha')}|\le\ell,\ldots,
    |V_{t_n}^{(\alpha')}|\le\ell
  \bigr)
\]
for every $\ell>0$.
Applying the same dyadic approximation and path-continuity argument
as in~(1), and then letting $\ell\uparrow2$, we obtain 
$
  \tau_{\alpha,\alpha}
  \le_{\mathrm{st}}
  \tau_{\alpha',\alpha'}.
$
This completes the proof.
\end{proof}

\Cref{prop:expect-tau} (1) can be reformulated directly in terms of
the original parameters $(\alpha_1,\alpha_2)$.
\begin{corollary}
\label{cor:exit-monotonicity-sum}
Fix $-1\le \mu<1/4$ and let $(\alpha_1,\alpha_2)$ and
$(\widetilde\alpha_1,\widetilde\alpha_2)$ be parameter pairs with $\alpha_1\alpha_2=\widetilde\alpha_1\widetilde\alpha_2=\mu$.
If $\left|\alpha_1+\alpha_2-1/2\right|\le\left|
\widetilde\alpha_1+\widetilde\alpha_2-1/2\right|$,
then 
\[
\tau_{\widetilde\alpha_1,\widetilde\alpha_2}
\le_{\mathrm{st}}
\tau_{\alpha_1,\alpha_2}.
\]
\end{corollary}
\begin{proof}
Let $\nu=\alpha_1+\alpha_2$ and $\widetilde \nu=\widetilde\alpha_1+\widetilde\alpha_2$.
For fixed $\mu<1/4$, the parameter $\lambda\in \Lambda_\mu$ is given by
\begin{equation}
    \label{eq:lambda_completing_square}
\lambda(\nu)
=
\frac{1}{1-4\mu}
\left\{
\frac14-\left(\nu-\frac12\right)^2
\right\}.
\end{equation}
Since $1-4\mu>0$, the assumption $\left|\nu-\frac12\right|\le
\left|\widetilde \nu-\frac12\right|$ implies $\lambda(\nu)\ge\lambda(\widetilde \nu)$.
The assertion therefore follows from part~(1) of
\Cref{prop:expect-tau}.
\end{proof}

The monotonicity established in \Cref{cor:exit-monotonicity-sum} allows us
to identify parameter pairs maximizing the expected limiting exit time on
each level set of $\mu=\alpha_1\alpha_2$. We do not claim uniqueness, since
\Cref{prop:expect-tau} establishes only non-strict monotonicity.
\begin{corollary}
\label{cor:exit-maximizer}
Fix $-1\le \mu<1/4$. On the level set $\mu=\alpha_1\alpha_2$,
one maximizer of
$m_{\alpha_1,\alpha_2}$, up to interchange of the coordinates, is
\[
(\alpha_1,\alpha_2)
=
\begin{cases}
(1,\mu),
&
-1\le \mu<-1/2,
\\[3mm]
\left(
\dfrac{1+\sqrt{1-16\mu}}{4},
\dfrac{1-\sqrt{1-16\mu}}{4}
\right),
&
-1/2\le \mu\le1/16,
\\[3mm]
(\sqrt{\mu},\sqrt{\mu}),
&
1/16<\mu<1/4.
\end{cases}
\]
The same pair maximizes every positive moment of the limiting exit time. 
\end{corollary}
\begin{proof}
By~\eqref{eq:lambda_completing_square} and
\Cref{prop:expect-tau}, it suffices to minimize
$
  \left|\alpha_1+\alpha_2-\frac12\right|
$
over $\alpha_1,\alpha_2\in[-1,1]$ subject to
$\alpha_1\alpha_2=\mu$.
If $-1\le\mu<-\frac12$, the largest possible value of
$\alpha_1+\alpha_2$ is $1+\mu<1/2$, attained at $(1,\mu)$ up to
interchange.
If $-\frac12\le\mu\le\frac1{16}$, the value $1/2$ is attainable, with
$
  (\alpha_1,\alpha_2)
  =
  \left(
    \frac{1+\sqrt{1-16\mu}}4,
    \frac{1-\sqrt{1-16\mu}}4
  \right).
$
Finally, if $\frac1{16}<\mu<\frac14$, the closest admissible value to
$1/2$ is $2\sqrt\mu$, attained at
$\alpha_1=\alpha_2=\sqrt\mu$.
The assertion follows.
\end{proof}

\begin{remark}
\label{remark:exit_time_monotone_max}
(i) By~\Cref{cor:exit-monotonicity-sum}, if $\alpha_2=0$, then $m_{\alpha_1,0}$ is nondecreasing in $\alpha_1$ on $[-1,1/2]$ and nonincreasing on $[1/2,1]$. By~\Cref{cor:exit-maximizer}, $m_{\alpha_1,0}$ is maximized at
    $\alpha_1=1/2$, in agreement with the numerical results shown in~\Cref{fig:expected_exit_time_curves_selfsimilar_fixmu}.

(ii) Consider the level set $\alpha_1\alpha_2=-1/4$.
For the three parameter pairs 
\[
(\alpha_1,\alpha_2)=\left(-1,1/4\right),\ \left(-1/2,1/2\right),\ \left(-1/4,1\right),
\]
the corresponding sums are $-3/4,0,3/4$, respectively. Hence \Cref{cor:exit-monotonicity-sum} yields
\[
m_{-1,1/4}
\le
m_{-1/2,1/2}
\le
m_{-1/4,1},
\]
which is consistent with our numerical simulations in~\Cref{fig:expected_exit_time_curves_selfsimilar_fixmu}.
Moreover, by~\Cref{cor:exit-maximizer}, $m_{\alpha_1,\alpha_2}$ is maximized at
\[
(\alpha_1,\alpha_2)
=
\left(
\frac{1+\sqrt5}{4},
\frac{1-\sqrt5}{4}
\right).
\]
\end{remark}

\begin{figure}[htbp]
    \centering
    \includegraphics[width=0.7\linewidth]{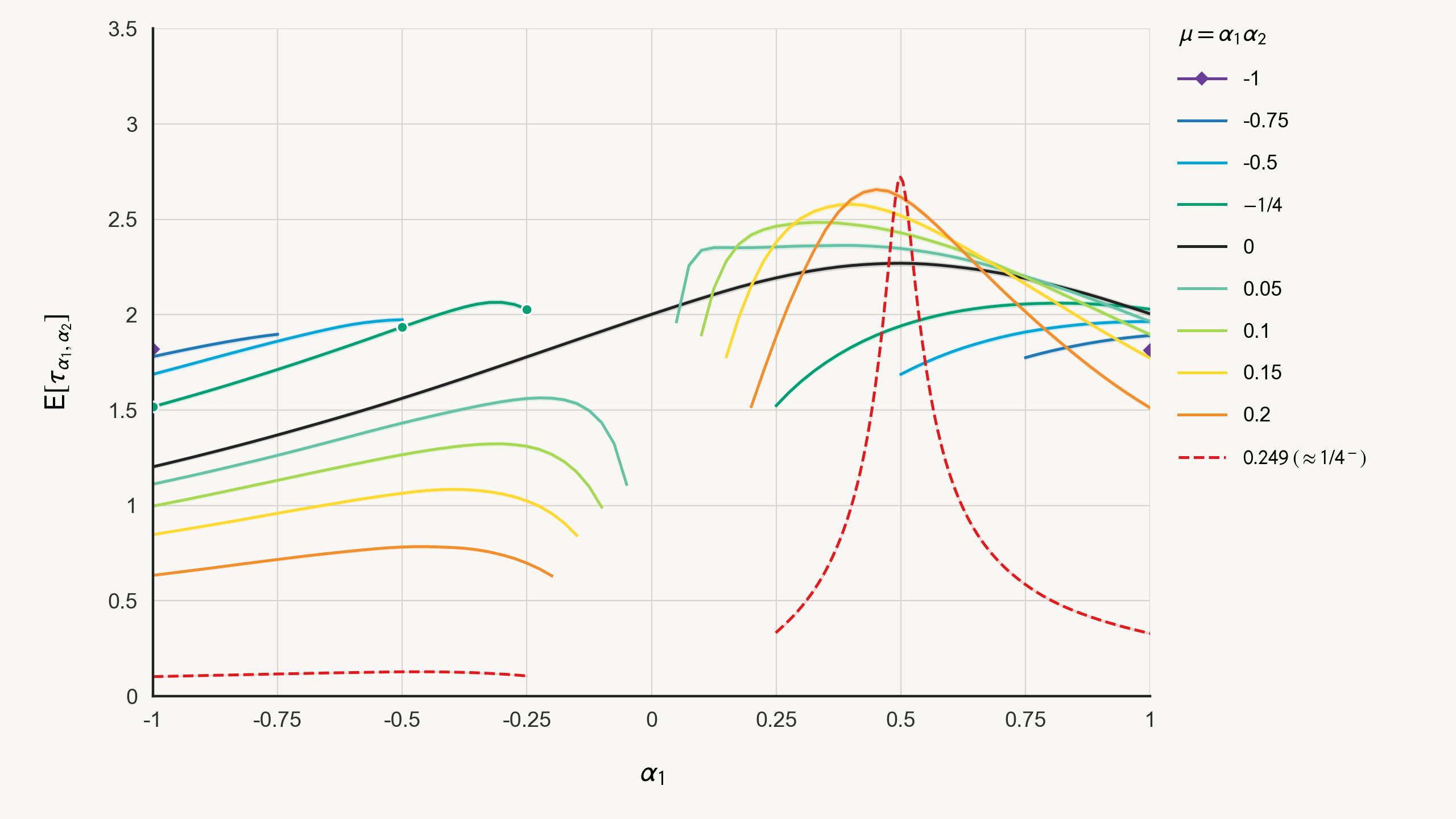}
    \caption{The expected exit time $m_{\alpha_1,\alpha_2}=\E[\tau_{\alpha_1,\alpha_2}]$
as a function of $\alpha_1$ for several fixed values of
$\mu=\alpha_1\alpha_2$.
Each coloured curve corresponds to a level set of $\mu$, with
$\alpha_2=\mu/\alpha_1$. For $\mu=0$ (we take $\alpha_2=0$, so that the second walk is a simple random walk), the graph is symmetric with respect to $\alpha_1=1/2$ on
$0\le\alpha_1\le1$, and satisfies
$m_{0,0}=m_{1,0}=2$.
The three marked points on the level set $\mu=-1/4$ correspond to the
parameter pairs considered in~\Cref{remark:exit_time_monotone_max}(ii), illustrating the
ordering of the expected exit times established there.
Finally, the dashed red curve corresponds to $\mu=0.249$, illustrating the
behaviour as $\mu\uparrow1/4$. 
$m_{\alpha_1,\alpha_2}$ remains finite at $\alpha_1=\alpha_2=1/2$ by \Cref{prop:exit-expectation}.}
    \label{fig:expected_exit_time_curves_selfsimilar_fixmu}
\end{figure}

\section*{Acknowledgments}
RA gratefully acknowledges the support of the Ongoing Research Funding Program
(ORF-2026-987), King Saud University, Riyadh, Saudi Arabia. SS was supported by the WISE program (MEXT) at Kyushu University, and was also supported in part by JP22H05105 and JP23K25774. TS was supported by JSPS KAKENHI Grant Numbers JP22H05105 and JP23K25774, and was partially supported by JP24KK0060 and JP26K22263.



\begin{thebibliography}{99}

\bibitem[ACG17]{aletti2017synchronization}
Giacomo Aletti, Irene Crimaldi, and Andrea Ghiglietti.
\newblock Synchronization of reinforced stochastic processes with a network-based interaction.
\newblock \emph{The Annals of Applied Probability}, 27(6):3787--3844, 2017.

\bibitem[AQ26]{aguech2026two}
Rafik Aguech and Shuo Qin.
\newblock How two elephants can learn from each other.
\newblock \emph{Stochastic Processes and their Applications}, 201:105040, 2026.

\bibitem[And55]{anderson1955integral}
Theodore~W. Anderson.
\newblock The integral of a symmetric unimodal function over a symmetric convex set and some
probability inequalities.
\newblock \emph{Proceedings of the American Mathematical Society}, 6(2):170--176, 1955.

\bibitem[AZ26]{andre2026estimates}
Morgan Andr\'e and Leonel Zuazn\'abar.
\newblock Estimates on escape times for the elephant random walk.
\newblock \emph{arXiv preprint arXiv:2602.18953}, 2026.

\bibitem[BB16]{baur2016elephant}
Erich Baur and Jean Bertoin.
\newblock Elephant random walks and their connection to P{\'o}lya-type urns.
\newblock \emph{Physical Review E}, 94(5):052134, 2016.

\bibitem[Ber18]{bercu2018martingale}
Bernard Bercu.
\newblock A martingale approach for the elephant random walk.
\newblock \emph{Journal of Physics A: Mathematical and Theoretical}, 51(1):015201, 2018.

\bibitem[Ber22]{bertenghi2022functional}
Marco Bertenghi.
\newblock Functional limit theorems for the multi-dimensional elephant random walk.
\newblock \emph{Stochastic Models}, 38(1):37--50, 2022. 

\bibitem[Ber20]{bertoin2020universality}
Jean Bertoin.
\newblock Universality of noise-reinforced Brownian motions.
\newblock In \emph{In and Out of Equilibrium 3: Celebrating Vladas Sidoravicius},
pages 147--161. Springer, 2020.

\bibitem[Bil99]{billingsley1999convergence}
Patrick Billingsley.
\newblock \emph{Convergence of Probability Measures}.
\newblock 2nd ed.\ Wiley, New York, 1999.

\bibitem[BL19]{bercu2019multidimensional}
Bernard Bercu and Lucile Laulin.
\newblock On the multi-dimensional elephant random walk.
\newblock \emph{Journal of Statistical Physics}, 175(6):1146--1163, 2019.

\bibitem[BRO22]{bertenghi2022joint}
Marco Bertenghi and Alejandro Rosales-Ortiz.
\newblock Joint invariance principles for random walks with positively and
negatively reinforced steps.
\newblock \emph{Journal of Statistical Physics}, 189:35, 2022.

\bibitem[CGS17a]{coletti2017strong}
Cristian~F. Coletti, Renato Gava, and Gunter~M. Sch\"utz.
\newblock A strong invariance principle for the elephant random walk.
\newblock \emph{Journal of Statistical Mechanics: Theory and Experiment}, 2017(12):123207, 2017.

\bibitem[CGS17b]{coletti2017central}
Cristian~F. Coletti, Renato Gava, and Gunter~M. Sch\"utz.
\newblock Central limit theorem and related results for the elephant random walk.
\newblock \emph{Journal of Mathematical Physics}, 58(5):053303, 2017.

\bibitem[Che14]{chen2014two}
Jun Chen.
\newblock Two particles' repelling random walks on the complete graph.
\newblock \emph{Electronic Journal of Probability}, 19(113):1--17, 2014.

\bibitem[Das24]{das2024elephant}
Deborshi Das.
\newblock Elephant random walks with graph based shared memory: First and second order asymptotics.
\newblock \emph{arXiv preprint arXiv:2410.22969}, 2024.

\bibitem[EK86]{ethier1986markov}
Stewart~N. Ethier and Thomas~G. Kurtz.
\newblock Markov Processes: Characterization and Convergence.
\newblock Wiley, New York, 1986.

\bibitem[ER24]{erhard2024stochastic}
Dirk Erhard and Guilherme Reis.
\newblock Stochastic processes with competing reinforcements.
\newblock \emph{The Annals of Applied Probability}, 34(5):4513--4553, 2024.

\bibitem[GLRS25]{guerin2025limit}
H\'el\`ene Gu\'erin, Lucile Laulin, Kilian Raschel, and Thomas Simon.
\newblock On the limit law of the superdiffusive elephant random walk.
\newblock \emph{Electronic Journal of Probability}, 30(102):1--25, 2025.

\bibitem[GLR26]{guerin2026fixed}
H\'el\`ene Gu\'erin, Lucile Laulin, and Kilian Raschel.
\newblock A fixed-point equation approach for the superdiffusive elephant random walk.
\newblock \emph{Annales de l'Institut Henri Poincar\'e, Probabilit\'es et Statistiques},
62(2):973--1005, 2026.

\bibitem[GMR24]{gantert2024interacting}
Nina Gantert, Fabian Michel, and Guilherme H. de Paula Reis.
\newblock Interacting edge-reinforced random walks.
\newblock \emph{ALEA, Latin American Journal of Probability and Mathematical Statistics},
21(2):1041--1072, 2024.

\bibitem[HH80]{hall1980martingale}
Peter Hall and C.~C. Heyde.
\newblock \emph{Martingale Limit Theory and Its Application}.
\newblock Academic Press, New York, 1980.

\bibitem[JS03]{jacod2003limit}
Jean Jacod and Albert~N. Shiryaev.
\newblock \emph{Limit Theorems for Stochastic Processes}.
\newblock 2nd ed.\ Springer, Berlin, 2003.

\bibitem[KT19]{kubota2019gaussian}
Naoki Kubota and Masato Takei.
\newblock Gaussian fluctuation for superdiffusive elephant random walks.
\newblock \emph{Journal of Statistical Physics}, 177(6):1157--1171, 2019.

\bibitem[Mar19]{marquioni2019multidimensional}
Vitor M. Marquioni.
\newblock Multidimensional elephant random walk with coupled memory.
\newblock \emph{Physical Review E}, 100(5):052131, 2019.

\bibitem[Qin26]{qin2026cover}
Shuo Qin.
\newblock Cover times and ranges of elephant random walks.
\newblock \emph{arXiv preprint arXiv:2609.17264}, 2026.

\bibitem[PCR23]{prado2023two}
Fernando P. A. Prado, Cristian F. Coletti, and Rafael A. Rosales.
\newblock Two repelling random walks on $\mathbb{Z}$.
\newblock \emph{Stochastic Processes and their Applications}, 160:72--88, 2023.

\bibitem[PR25]{prado2025interacting}
Fernando P. A. Prado and Rafael A. Rosales.
\newblock Interacting vertex reinforced random walks on complete sub-graphs.
\newblock \emph{arXiv preprint arXiv:2508.15992}, 2025.

\bibitem[Pro05]{philip2005stochastic}
Philip~E. Protter.
\newblock \emph{Stochastic Integration and Differential Equations}.
\newblock 2nd ed., version 2.1, Springer, Berlin, 2005.

\bibitem[RM26]{ramirez2026frog}
J.~H. Ram\'irez-Gonz\'alez and F\'abio Prates Machado.
\newblock The frog model on $\mathbb{Z}$ with random discrete Weibull lifetimes and elephant random walks.
\newblock \emph{ResearchGate preprint}, 2026.
\newblock doi:10.13140/RG.2.2.28211.18724.

\bibitem[RPP22]{rosales2022vertex}
Rafael A. Rosales, Fernando P. A. Prado, and Benito Pires.
\newblock Vertex reinforced random walks with exponential interaction on complete graphs.
\newblock \emph{Stochastic Processes and their Applications}, 148:353--379, 2022.

\bibitem[RTT24]{roy2024often}
Rahul Roy, Masato Takei, and Hideki Tanemura.
\newblock How often can two independent elephant random walks on $\Z$ meet?
\newblock \emph{Proceedings of the Japan Academy, Series A, Mathematical Sciences}, 100(10):57--59, 2024.

\bibitem[Shi25]{shibata2025functional}
Shuhei Shibata.
\newblock Functional limit theorems for elephant random walks on general periodic structures.
\newblock \emph{arXiv preprint arXiv:2511.10347}, 2025.

\bibitem[SS25]{shibata2025remark}
Shuhei Shibata and Tomoyuki Shirai.
\newblock A remark on elephant random walks via the classical law of the iterated logarithm for
self-similar Gaussian processes.
\newblock \emph{Illinois Journal of Mathematics}, 69(3):567--582, 2025.

\bibitem[ST04]{schutz2004elephants}
Gunter~M. Sch\"utz and Steffen Trimper.
\newblock Elephants can always remember: Exact long-range memory effects in a non-Markovian random walk.
\newblock \emph{Physical Review E}, 70(4):045101(R), 2004.

\bibitem[VW23]{wellner2023weak}
A.~W.\ van der Vaart and Jon~A. Wellner.
\newblock \emph{Weak Convergence and Empirical Processes: With Applications to Statistics}.
\newblock 2nd ed., Springer, Cham, 2023.

\bibitem[Whi07]{whitt2007proofs}
Ward Whitt.
\newblock Proofs of the martingale FCLT.
\newblock \emph{Probability Surveys}, 4:268--302, 2007.

\end{thebibliography}
\end{document}